\documentclass[twoside, reqno, 10pt]{amsart}

\usepackage[left=3.8cm, right=3.8cm, top=3cm, bottom=2.5cm]{geometry}

\usepackage{algorithm}
\usepackage{algpseudocode}
\usepackage{placeins}
\usepackage[colorlinks=true, pdfstartview=FitV, linkcolor=blue,citecolor=blue, urlcolor=blue]{hyperref}

\usepackage[usenames]{xcolor} %I put this in to get rid of the incompatible color warnings -ken

\definecolor{labelkey}{rgb}{0,0,1}
\definecolor{Red}{rgb}{0.7,0,0.1}
\definecolor{Green}{rgb}{0,0.7,0}

\usepackage{amsfonts, amssymb, amsmath, amsthm, mathrsfs, bbm, cjhebrew, gensymb, textcomp, mathtools, dsfont, calligra}

\usepackage[normalem]{ulem}%normalem gets rid of underlining of emph

\usepackage{commath, setspace, subcaption, parcolumns, multirow, multicol, accents, comment, marginnote, verbatim, empheq, enumerate, stackrel, enumitem, float, physics}
\usepackage[capitalize,nameinlink,noabbrev]{cleveref}

\usepackage{graphicx, graphics, epsfig, psfrag, tikz, tikz-cd,svg}

\numberwithin{equation}{section}

\usepackage{tikz}

\newtheorem{Thm}{Theorem}[section]
\newtheorem{Lem}[Thm]{Lemma}
\newtheorem{Prop}[Thm]{Proposition}
\newtheorem{Cor}[Thm]{Corollary}

\newtheorem{Rmk}[Thm]{Remark}

\newtheorem*{Thm*}{Theorem}

\newcommand{\ZZ}{\mathbb{Z}}

\newcommand{\cO}{\mathcal{O}}

\newcommand{\no}[2]{\lVert#2\rVert_{#1}}
\newcommand{\Abs}[2]{\lvert#2\rvert_{#1}}

\newcommand{\goesto}{\rightarrow}
\newcommand{\smod}{\setminus}

\newcommand{\de}{\delta}
\newcommand{\De}{\Delta}

\newcommand{\eps}{\epsilon}
\newcommand{\veps}{\varepsilon}

\newcommand{\si}{\sigma}

\newcommand{\kap}{\kappa}

\newcommand{\Om}{\Omega}

\newcommand{\bdy}{\partial}

\newcommand{\lp}{\left(}
\newcommand{\rp}{\right)}
\newcommand{\lpp}{(\!(}
\newcommand{\rpp}{)\!)}

\newcommand{\tu}{\tilde{u}}

\newcommand{\tw}{\tilde{w}}

\newcommand{\tv}{\tilde{v}}

\newcommand{\tq}{\tilde{q}}

\newcommand{\tp}{\tilde{p}}
\newcommand{\tC}{\tilde{C}}

\newcommand{\Gr}{\mathfrak{G}}

\newcommand{\sS}{\mathscr{S}}
\newcommand{\tsS}{\tilde{\mathscr{S}}}
\newcommand{\Cp}{\mathfrak{p}}
\newcommand{\Cq}{\mathfrak{q}}

\usepackage{todonotes}
\usepackage{amssymb}
\usepackage{multicol}
\usepackage[normalem]{ulem}
 
 \title[Infinite Error Feedback Gain]{
 On the nudging approach to continuous data assimilation in the limit of infinite error feedback gain
 }
 \author{Elizabeth Carlson, Aseel Farhat, Vincent R. Martinez, Collin Victor}
\begin{document}

\begin{abstract}
This article studies the intimate relationship between two filtering algorithms for continuous data assimilation, the {direct-replacement algorithm} and the {nudging algorithm}, in the paradigmatic context of the two-dimensional (2D) Navier-Stokes equations (NSE) for incompressible fluids. In this setting, the nudging algorihm can formally be viewed as an affine perturbation of the 2D NSE. Thus, in the degenerate limit of zero nudging parameter, the {nudging algorithm} correspondingly converges to the solution of the 2D NSE. However, when the nudging parameter of the {nudging algorithm} is made increasingly large, the perturbation apparently becomes more singular. It is nevertheless shown here that in the limit of infinite nudging parameter, the nudging filter converges to the {direct-replacement algorithm}. In establishing this result, the article fills a notable gap in the literature surrounding these algorithms. Numerical experiments are then presented that confirm the theoretical results and probes the issue of selecting a nudging strategy in the presence of observational noise. In this direction, an adaptive nudging strategy is proposed that leverages the insight gained from the relationship between the {direct-replacement algorithm} and the {nudging algorithm} that produces measurable improvement over the constant nudging strategy.
\end{abstract}

\maketitle

{\noindent \small {\it {\bf Keywords: continuous data assimilation, {nudging algorithm}, {direct-replacement algorithm}, Navier-Stokes equations, infinite error feedback gain, adaptive nudging}
  }  
}

%\tableofcontents
%\setcounter{tocdepth}{2}

\section{Introduction}\label{sect:intro}

Data assimilation (DA) was born in the 1960s, when it was proposed by J. Charney, M. Halem, and R. Jastrow \cite{CharneyHalemJastrow1969} that the equations of motion of the atmosphere be used to process observations collected on the evolving state of the atmosphere for the purpose of improving their prognostic capabilities. Preceding \cite{CharneyHalemJastrow1969}, it was proposed in a milestone paper of V. Bjerknes \cite{Bjerknes1904} that the problem of weather prediction be reduced to the direct numerical simulation of the equations of motion and the obtaining of a sufficiently accurate approximation of the state of the atmosphere with which to initialize the equations. It was in the advent of scientific computing in the 1950s \cite{CharneyFjortoftVonNeumann1950, Charney1951} and launch of the first weather satellites in the 1960s \cite{KalnayBook} that DA was conceived in the spirit of the mechanistic perspective to meteorology of Bjerknes. Although numerical weather prediction continues to be an important application of DA, DA methods have since become essential in any situation for which both a model and observations on the modeled phenomenon are available. 

Two fundamental issues in the study of DA derive from the nonlinear and often high-dimensional nature of the models of interest, as well as the presence of errors in both the model and observations. Since the seminal work of R.E. Kalman and R.S. Bucy \cite{Kalman_Bucy_1961}, it has been known that when the system is linear and errors are Gaussian, the optimal predictive process is also linear and Gaussian. Thus, any situation in which the underlying system is nonlinear must naturally contend with non-Gaussianity, leading to issues in sampling and efficient computation. Perhaps more importantly, the results in \cite{Kalman_Bucy_1961} indicated that new methods were required to account for nonlinearity appropriately. Many important efforts have subsequently been dedicated to addressing these issues; the reader is referred to the following seminal works \cite{Lorenc1986, Talgrand_LeDimet_1986, GordonSalmondSmith1993, Evensen_1997, BurgersvanLeeuwenEvensen1998} and the review articles \cite{Kunsch_2013, Carrassi_et_al_2018, Pandya_et_al_2022}. On the other hand, in many applications, the model of interest is often given by a nonlinear system of partial differential equations (PDEs). Thus, a deeper understanding of DA methods is inevitably rooted in the understanding of how the incorporation of observations interact with nonlinearity in PDEs; this article is a contribution to the latter endeavor.

In particular, the primary concern of this paper is to establish a rigorous theoretical relationship between two filtering algorithms for CDA, the \textit{{direct-replacement algorithm}} and the \textit{{nudging algorithm}}. The study of these algorithms is carried out under the mathematically ideal assumptions of the availability of a perfect model, given by a dissipative partial differential equation (PDE), and noise-free observations, so that the fundamental issues of nonlinearity and high-dimensionality are isolated from the issues involved with the presence of noise. Naturally, the important issue of the effect of observational or model errors must be addressed in a subsequent study. This work may nevertheless be considered as a foundational step in bridging our understanding of these two algorithms. In this direction, the reader is also referred to a recent work of the authors \cite{CarlsonFarhatMartinezVictor2025a, CarlsonFarhatMartinezVictor2025b} in which the precise relationship between the determining modes property, {direct-replacement algorithm}, and {nudging algorithm} is established through the notion of \textit{intertwinement}.

The particular model we consider is the Navier-Stokes equations (NSE) for incompressible fluids in two-dimensions (2D), which has been used as a paradigmatic example for CDA studies \cite{ApteJonesStuartVoss2008, Sanz-AlonsoStuart2015} due to its connection as a model for turbulent fluid flow, a phenomenon that exhibits a large number of degrees of freedom and chaotic dynamical behavior \cite{FoiasTemam1987, FoiasManleyRosaTemamBook2001}. %The important issue of the effect of observational errors is not addressed here as this work constitutes a first step in establishing unifying one's understanding of these two filters.
We will specifically consider the 2D NSE over a rectangular spatial domain, $\Om=[0,2\pi]^2$, equipped with periodic boundary conditions for analytic convenience:
    \begin{align}\label{eq:nse}
        \bdy_tu+u\cdotp\nabla u=-\nabla p+\nu\De u+f,\quad \nabla\cdotp u=0.
    \end{align}
where $u=(u^1,u^2)$ denotes the velocity vector field, $\nu$ the kinematic viscosity, $p$ the scalar pressure field, and $f$ a given external force {which can be thought of a source of energy that sustains} turbulent {or chaotic} behavior of the modeled fluid. We denote solutions to the corresponding initial value problem of \eqref{eq:nse} by $u(\cdotp;u_0)$, where $u(0;u_0)=u_0$. In the analysis of \eqref{eq:nse}, it is customary to apply the Leray projection onto divergence-free vector fields to \eqref{eq:nse} and subsequently consider the equivalent functional formulation of \eqref{eq:nse} given as
    \begin{align}\label{eq:nse:ff}
        \frac{du}{dt}+\nu Au+B(u,u)=Pf,
    \end{align}
where $P$ denotes the Leray projection onto divergence-free vector fields, $A=-P\De$ is the Stokes operator, and $B$ is the bilinear form defined by
    \begin{align}\label{def:B}
        B(u,v):=P(u\cdotp\nabla)v.
    \end{align}
If we assume that $f$ is divergence-free, then $Pf=f$; this will be a standing assumption henceforth. 

As previously mentioned, \eqref{eq:nse} is assumed to be our representation of reality. Under this assumption, the observations collected on the underlying reality are formulated as a continuous time-series
    \begin{align}\label{def:obs}
        \cO_N(u_0)=\{P_Nu(t;u_0)\}_{t\geq0},
    \end{align}
where $N\geq0$ is a real number and $P_N$ denotes projection onto Fourier wavenumbers $|k|\leq N$. In particular, the exact values of $(I-P_N)u(t)$ are \textit{unknown} for all $t\geq0$. We will denote the complementary projection by
    \begin{align}\label{def:QN}
        Q_N:=I-P_N.
    \end{align}  
We will also make use of the shorthand notations $\cO$ in place of $\cO_N(u_0)$ and ${{}Q}$ in place of $Q_Nu$, particularly when the context makes clear the dependence on $u_0$ and $N$.

The {direct-replacement algorithm} is defined by directly inserting the observations into the system, then integrating the subsequent equation forward-in-time to obtain an approximation of the unobserved state variables.  This is effectively the manner in which the observations are to be processed by the equation that was proposed in \cite{CharneyHalemJastrow1969}. {It may be seen as a particular approach to synchronization of dynamical systems, as well as a type of ``reduced-order observer" in which only the unobserved variables are modeled dynamically; this latter notion was introduced by D.G. Luenberger in \cite{Luenberger1964}. To precisely define the {direct-replacement algorithm}, let}
    \begin{align}\label{def:p}
        p:=P_Nu(\cdotp;u_0),
    \end{align}
where $u(\cdotp;u_0)$ satisfies \eqref{eq:nse:ff}. Then the \textit{{direct-replacement algorithm}} is defined as 
    \begin{align}\label{def:sync}
        v:=p+q,\quad q=\sS(p;q_0),
    \end{align}
where $\sS$ denotes the solution operator to the following initial value problem:
    \begin{align}\label{eq:sync}
        \frac{dq}{dt}+\nu Aq+Q_NB(p+q,p+q)=Q_Nf,\quad q(0)=q_0.
    \end{align}
We will also make use of the expanded notation $v=v(\cdotp;v_0, \cO_N(u_0))$ for \eqref{def:sync}, where it is implicitly assumed that $v_0=P_Nu_0+q_0$. It is important to note that $q_0$ is not necessarily equal to ${{}{{}Qu_0}=Q_Nu_0}$. Indeed, when ${{}q_0=Qu_0}$, then $v=u$, or equivalently that
    \begin{align}\notag
        v=p+\sS(p;{{}{{}Qu_0}})=u(\cdotp;u_0).    
    \end{align}
{{}The direct-replacement algorithm was originally}  studied by E. Olson and E.S. Titi \cite{OlsonTiti2003} in the context of the 2D NSE, under the same mathematically ideal assumptions described above. In \cite{OlsonTiti2003}, it was shown that the algorithm successfully reconstructs the unobserved state, namely
    \begin{align}\label{eq:OT:result}
        \lim\limits_{t\goesto\infty}\|v(t;v_0,\cO_N(u_0))-u(t;u_0)\|_{L^2}=0, 
    \end{align}
provided that $N=N(\nu,f)$ is sufficiently large. In other words, the {direct-replacement algorithm} successfully reconstructs the unobserved state variables asymptotically in time provided that sufficiently many {length scales of the flow} are observed for all time. 

This algorithm, applied to the 2D NSE, was studied numerically in \cite{OlsonTiti2008}, in its discrete-in-time formulation by \cite{HaydenOlsonTiti2011}, and later in the presence of unbounded observational noise by \cite{OljacaBrockerKuna2018}; a related work that preceded \cite{OljacaBrockerKuna2018} is \cite{BrettLamLawMcCormickScottStuart2012}, which considers the case of bounded observational noise. Other works that improve upon the algorithm in different directions include \cite{CelikOlsonTiti2019}, which expand the algorithm to include non-spectral observations, such as volume element or nodal value observations, and \cite{CelikOlson2023}, where a mechanism to filter observational noise is incorporated. It is notable that \cite{OlsonTiti2003} is one of the first works to study CDA filters for nonlinear partial differential equations via rigorous mathematical analysis. One of the key insights from \cite{OlsonTiti2003} is that the success of CDA in the context of the 2D NSE and related equations is the presence of a nonlinear mechanism for asymptotically enslaving small scales to large scales. This mechanism was originally discovered in the context of the 2D NSE by C. Foias and G. Prodi in \cite{FoiasProdi1967} as the property of having finitely many \textit{determining modes}. Subsequent works found several different forms of this property \cite{FoiasTemam1984, CockburnJonesTiti1997, JonesTiti1992a, JonesTiti1992b}, which in turn formed the mathematical justification of many studies in CDA. Among these is the study of another elemental CDA filter known as the {nudging algorithm}.

The {nudging algorithm} is defined by inserting the observed state exogenously into the system of interest as a feedback control term that serves to guide the state towards that of the observations, but only the subspace where observations are available. The approximating state of the system is then produced by integrating the controlled equation forward-in-time. In our setting, the \textit{{nudging algorithm}} can be defined more precisely as
    \begin{align}\label{def:nudge}
        \tv:={\tsS(p;\tv_0)}
        %\tp+\tsS(\tp;p),\quad \tq=\tsS(\tp;p),
    \end{align}
%where 
%	\[
%		\tp:=P_N\tv,\quad \tq:=Q_N\tv,
%	\]
where {$\tsS$ denotes the solution operator to the following initial value problem:}
    \begin{align}\label{eq:nudge}
        \frac{d\tv}{dt}+\nu A\tv+B(\tv,\tv)=f-\mu {\tp}+\mu {p},\quad \tv(0)=\tv_0,
    \end{align}
where {$\tp:=P_N\tv$ and $p$ is given by \eqref{def:p}}. %$u=u(\cdotp;u_0)$ satisfies \eqref{eq:nse:ff}. 
We {will also write} the solution to \eqref{eq:nudge} as $\tv=\tv(\cdotp;\tv_0,\cO_N(u_0))$ {in order to exhibit the dependence of $\tv$ on its initial data and the available observations}. Note that in general ${P_N\tv_0}$ need not equal to ${P_Nu_0}$, but if $\tv_0=u_0$, then $\tv=u$. We will refer to \eqref{def:nudge}, \eqref{eq:nudge} as the \textit{{nudging algorithm}}. It was shown by A. Azouani, E. Olson, and E.S. Titi \cite{AzouaniOlsonTiti2014} that the {nudging algorithm} successfully reconstructs the unobserved state variables asymptotically in time in the sense that
    \begin{align}\label{eq:AOT:result}
        \lim\limits_{t\goesto\infty}\|\tv(t;\tv_0,\cO_N(u_0))-u(t;u_0)\|_{L^2}=0,
    \end{align}
provided that sufficiently many {length-scales of the flow} are observed for all time and that the nudging parameter is accordingly tuned. 

{The use of nudging for the purpose of numerical weather prediction was first proposed} by J.E. Hoke and R.A. Anthes \cite{HokeAnthes1976}, although their study was restricted to the setting of finite-dimensional systems of ordinary differential equations. Many studies on nudging and synchronization-based techniques for data assimilation have since followed these classical works, but mostly in the setting of nonlinear systems of ODEs \cite{ZouNavonLedimet1992, AurouxBlum2008, PazoCarrassiLopez2016, PinheirovanLeeuwenGeppert2019}. However, in the seminal work \cite{DuaneTribbiaWeiss2006}, it was recognized that the ability of nonlinear systems to intrinsically synchronize \cite{PecoraCarroll1990} could be facilitated through nudging and therefore leveraged for the purposes of DA even in PDEs. The work \cite{AzouaniOlsonTiti2014} was one of the first to study the {nudging algorithm} in the context of partial differential equations in a mathematically rigorous fashion. One of the main achievements of \cite{AzouaniOlsonTiti2014} was to properly recognize the flexibility of the feedback control term to accommodate a large class of observation-types, {particularly} other than spectral observations. Indeed, it is shown in \cite{AzouaniOlsonTiti2014} that \eqref{eq:AOT:result} still holds if $P_N$ is replaced with a linear operator $I_h$ satisfying suitable approximation properties.

Numerical experiments analogous to those carried out in \cite{OlsonTiti2008} for the {direct-replacement algorithm} were carried out in \cite{GeshoOlsonTiti2016} for the {nudging algorithm}, and explored further in several subsequent works \cite{AltafTitiGebraelKnioZhaoMcCabeHoteit2017, FarhatJohnstonJollyTiti2018, DesamsettiDasariLangodanTitiKnioHoteit2019, ClarkDiLeoniMazzinoBiferale2020, BuzzicottiClarkDiLeoni2020}. In the presence of observational noise, studies were carried out by D. Bl\"omker, K. Law, A. Stuart, and K. Zygalakis \cite{BlomkerLawStuartZygalakis2013}, but only in the context of spectral observations, and H. Bessaih, E. Olson, and E.S. Titi \cite{BessaihOlsonTiti2015} in the more general framework of \cite{AzouaniOlsonTiti2014}. Notably, in the presence of observational noise, the {nudging algorithm} can be viewed as a suboptimal estimation of the mean of the state, in contrast to the 3DVAR filter, which provides updates in an optimal way {as an approximate Gaussian filter} (see, for instance, \cite[Equation 16]{BlomkerLawStuartZygalakis2013} in contrast with \cite[Equation 21]{BessaihOlsonTiti2015}), thus giving a logical primacy to the study of the {nudging algorithm}, as it forms the analytical core of the more sophisticated optimized setup. Because of this, the {nudging algorithm} has enjoyed a wealth of activity since \cite{AzouaniOlsonTiti2014}. It has been used as framework to give mathematical justification to typical practices in DA and its in many hydrodynamic or geophysical scenarios  \cite{FarhatJollyTiti2015, FarhatLunasinTiti2016a, FarhatLunasinTiti2016b, FarhatLunasinTiti2016c, AlbanezLopesTiti2016,   FoiasMondainiTiti2016, BiswasMartinez2017, JollyMartinezTiti2017,  AlbanezBenvenutti2018, BiswasFoiasMondainiTiti2018, BlocherMartinezOlson2018, JollyMartinezOlsonTiti2019, FarhatGlattHoltzMartinezMcQuarrieWhitehead2020, BiswasBradshawJolly2021, Franz_Larios_Victor_2021, CaoGiorginiJollyPakzad2022, You2024, Biswas_Branicki_2024}. It has also found application to improving numerical approximation \cite{MondainiTiti2018, IbdahMondainiTiti2019, ZerfasRebholzSchneierIliescu2019, LariosRebholzZerfas2019, HammoudTitiHoteitKnio2022, JollyPakzad2023, Garcia-ArchillaLiNovoRebholz2024}, inverse problems \cite{DiLeoniClarkMazzinoBiferale2018,CarlsonHudsonLarios2020, CarlsonHudsonLariosMartinezNgWhitehead2021, PachevWhiteheadMcQuarrie2022, Martinez2022, BiswasHudson2023, Martinez2024, FarhatLariosMartinezWhitehead2024, AlbanezBenvenutti2024}, and the study of long-time dynamics of various nonlinear PDEs \cite{FoiasJollyKravchenkoTiti2012, FoiasJollyKravchenkoTiti2014, JollySadigovTiti2015, JollySadigovTiti2017, FoiasJollyLithioTiti2017, JollyMartinezSadigovTiti2018}.

In spite of these many recent developments, the exact relationship between the {direct-replacement algorithm}, {nudging algorithm}, and underlying dynamical equation has remained a folklore result in the DA community. This relationship is rigorously addressed in the present article by considering the singular infinite-nudging limit ($\mu\goesto\infty$) in the {nudging algorithm} within the paradigmatic setting of the 2D NSE. In particular, the following convergence result is established: Let $H$ denote the subspace of square-integrable, divergence-free vector fields over $\Om$, which are mean-free and $2\pi$-periodic in each direction, and let $V$ denote the subspace of $H$ endowed with the topology of $H^1$. Then

\begin{Thm}\label{thm:main}
Given $f\in L^\infty(0,\infty;H)$ and $u_0\in V$, let $u$ denote the unique solution to the initial value problem corresponding to \eqref{eq:nse:ff}. {{}Suppose that the nudging and direct-replacement algorithms are initialized identically.} For any {$N>0$} {and $T>0$, one has}
       \begin{align}\label{eq:main:claim2}
        \lim\limits_{\mu\goesto\infty}\sup\limits_{t\in[0,T]}\|\tv(t;v_0,\cO_N(u_0))-v(t;v_0,\cO_N(u_0))\|_{L^2}=0,
        \end{align}
{for all $v_0\in V$ {{}such that $P_Nv_0=P_Nu_0$}.
%whereas if $Q_N\tv_0=Q_Nv_0$, but $P_N\tv_0\neq P_Nv_0$, then for any $\tv_0\in V$
%     \begin{align}\label{eq:main:claim2'}
        %\lim\limits_{t_0\goesto0^+}\limsup\limits_{\mu\goesto\infty}\sup\limits_{t\in[{t_0},T]}\|\tv(t;{\tv_0},\cO_N(u_0))-v(t;v_0,\cO_N(u_0))\|_{L^2}=0.
        %\end{align}
Moreover, there exists $N_*$ such that for all $N\geq N_*$
    \begin{align}\label{eq:main:claim3}
        \lim\limits_{\mu\goesto\infty}\sup\limits_{t\geq0}\|\tv(t;v_0,\cO_N(u_0))-v(t;v_0,\cO_N(u_0))\|_{L^2}=0.
    \end{align}
%and, if $Q_N\tv_0=Q_Nv_0$, but $P_N\tv_0\neq P_Nv_0$, then
%    \begin{align}\label{eq:main:claim3'}
        %\lim\limits_{t_0\goesto0^+}\limsup\limits_{\mu\goesto\infty}\sup\limits_{t\geq t_0}|\tv(t;\tv_0,\cO_N(u_0))-v(t;v_0,\cO_N(u_0))|=0.
    %\end{align}
}
\end{Thm}

A simple heuristic that quickly reveals the relationship between the {nudging algorithm} and {direct-replacement algorithm} is to simply divide by $\mu$ in \eqref{eq:nudge} and then pass to the limit as $\mu\goesto\infty$:
    \begin{align}\notag
        -P_N\tv+P_Nu=\frac{1}{\mu}\left(\frac{d\tv}{dt}+\nu A\tv+B(\tv,\tv)-f\right)\goesto0.
    \end{align}
Thus, $P_N\tv=P_Nu$ is enforced in the infinite-nudging regime. The issue with this formal argument is that without additional assumptions on $N$, the {a priori} analysis of $\tv$ produces bounds that \textit{depend linearly} on $\mu$ {(see \eqref{est:v:apriori} in \cref{rmk:main})}. {Moreover, the above heuristic, while it implicitly suggests that the high-modes of $\tv$ in the infinite-nudging limit would evolve according to \eqref{eq:sync}, it does not indicate what happens if \eqref{eq:nudge} and \eqref{eq:sync} are initialized differently. On the other hand, it was shown in \cite{OlsonTiti2003} that for $N$ sufficiently large, $v\goesto u$ as $t\goesto\infty$, while in \cite{AzouaniOlsonTiti2014}, for $N$ sufficiently large and $\mu$ sufficiently large, \textit{but bounded suitably by $N$}, then $\tv\goesto u$ as $t\goesto\infty$. Within this regime of parameters, $\tv$ therefore converges to $v$ as $t\goesto\infty$, since they both asymptotically converge to $u$. However, this suggests that even if one allows for $N$ to be taken sufficiently large, that one is unable to pass to the infinite nudging limit without also passing to the infinite $N$ limit. The main challenge is therefore to {{}develop an analysis} that allows one to overcome the {{}two} following \textit{apparent} obstructions in passing to the {{}infinite nudging limit}: 1) {{}the apriori bounds $|\tv|^2 \leq O(\mu)$}, and  2) {{}the relation that $\lim_{\mu\goesto\infty}N=\infty$ implied by the joint condition imposed in \cite{AzouaniOlsonTiti2014} to ensure asymptotic convergence of the nudging algorithm. We furthermore address} the case of \textit{incompatible} initial data, when the low-modes of the nudging and direct-replacement algorithm are not equal, but the high-modes are equal; the challenge presented here is to establish suitable control on the transient errors induced by the initial incompatibility (see \eqref{est:nudge:infinity:positive} of \cref{thm:nudge:limit} and \eqref{eq:nudge:complete:positive} of \cref{cor:nudge:complete}). %and 3) potential incompatibility of initial conditions between the nudging and direct-replacement algorithms.

In \cref{sect:limit}, we develop such {{}an analysis} and prove \cref{thm:main} by recognizing that {{}low-mode errors are in fact driven to zero as the nudging parameter becomes unboundedly large} on \textit{any finite-time horizon}. However, since the underlying system is nonlinear, one must then negotiate the coupling between high-modes and low-modes. This issue is overcome by establishing a Lipschitz relationship between the mapping from low-modes to high-modes in a suitable trajectory space. This Lipschitz property is ultimately afforded by {{}the freedom to lose as many derivatives as needed} on low-mode projections, a maneuver which is counter-intuitive to all previous analyses surrounding the nudging algorithm. {{}Indeed, doing so would naturally force} $\mu$ to be large relative to $N$. {{}On the other hand, in light of the joint condition identified by \cite{AzouaniOlsonTiti2014} which has $\lim_{\mu\goesto\infty}N=\infty$, one would subsequently be lead to expect} that $N$ {{}should} also be taken large {{}in order} to ensure asymptotic convergence of $\tv$ to $u$. The crucial observation made here is that the usual condition that forces $\mu$ to be bounded by $N$ \textit{can be relaxed} by sacrificing the faster rate of asymptotic convergence {{}provided by $\mu$ for the typically slower rate determined by $\nu$}. In doing so, we provide three important clarifications to the analysis of the nudging and direct-replacement algorithms: 1) on arbitrarily long finite-time horizons, the nudging parameter can be taken arbitrarily large, and allowed to depend on \textit{both} the size of the time-window and the {{}spatial} observational density,  2) one need not take $N$ sufficiently large relative to $\mu$ to ensure asymptotic convergence of $\tv$ to $u$, and 3) initial incompatibilities can be compensated for in arbitrarily short time simply by continuity. In this way, our paper not only establishes a basic and fundamental connection between the nudging and direct-replacement algorithms for CDA, but it gives clarity to the precise interrelation between the time-windows of convergence, the observational density, and the nudging parameter.
}
%We show that such a relationship is possible by making a trivial, but nevertheless crucial observation that the low-mode projection of any function is smooth, 

%and provide a proof of \cref{thm:main}; it relies crucially on the interplay between the {large-scale} stabilizing mechanism of the observations in the {nudging algorithm} and the continuity properties of the {low-mode to high-mode mapping that induces the} {direct-replacement algorithm}. {The latter property is a genuinely nonlinear property since it allows one to control discrepancies in the high-mode evolution by discrepancies in the low-mode evolution on arbitrary finite-time horizons. It is important to note that such a property fails to hold for the linear evolutionary Stokes equation wherein each mode evolves independently of one another. In this way, the validity of \cref{thm:main} is facilitated by the presence of nonlinearity.}
%is related to the so-called squeezing property of the 2D NSE (see, for instance, \cite{RobinsonBook, Temam_1997_IDDS}), a genuinely nonlinear property of the system. 

In contrast to the infinite-nudging limit, the complementary limit of zero-nudging parameter is degenerate in the sense that the {nudging algorithm} collapses back to a solution of the 2D NSE initialized with the same initial value of the {nudging algorithm}. Namely, one has
    \begin{align}\label{eq:main:claim1}
        \lim\limits_{\mu\goesto0}\sup\limits_{t\in[0,T]}\|\tv(t;\tv_0,\cO_N(u_0))-u(t;\tv_0)\|_{L^2}=0,
    \end{align}
for any $\tv_0\in V$ {and any $T>0$}. In this regime, all information from the observations $\cO_N(u_0)$ is {constrained to the initial data}. In effect, the zero-nudging limit collapses to what one might call the ``Bjerknes algorithm," which simply integrates \eqref{eq:nse} forward-in-time with whatever initial condition one managed to generate offline. For the sake of completeness, a proof of \eqref{eq:main:claim1} is provided in \cref{sect:limit:zero}. {In doing so, we {{}rigorously locate} the nudging algorithm as an intermediary between the the Bjerknes algorithm and direct-replacement algorithm, {{}and ultimately establish a deeper connection between the seminal results of \cite{OlsonTiti2003} and \cite{AzouaniOlsonTiti2014}.}

The paper concludes with \cref{sect:numerical} where we present the results of a variety of systematic numerical experiments that confirm the infinite-nudging limit, as well as the zero-nudging limit. The results of these experiments naturally lead one to consider intermediate possibilities between these two limiting regimes by allowing $\mu$ to be state-dependent, but constrained to the information available from the observations. Upon inspecting how the error dynamics transition from one regime to the next, we identify a simple adaptive scheme that measurably improves upon the constant-$\mu$ strategy in light of the analytical results  obtained in \cite{BessaihOlsonTiti2015} {{}in the presence of noisy observations}.

%{Include zero nudging limit}

%{Comment on Bjerknes}

%{Comment on Leo's result}

%{Intertwinement result}

%{Discussion on numerical results}
%; mention Layton result (MORE DETAILED REMARK)}

\section{Mathematical Preliminaries}\label{sect:notation}

%We let $H$ denote the space of $L^2$ real-valued vector fields, which are $2\pi$-periodic in each direction, divergence-free, and mean-free over $\Om$, in the sense of distribution. We let $\PP$ denote the Leray projection. Observe that $\PP H=H$. We let $V$ denote the subspace of $H$ endowed with the $V$ topology. 
We denote the inner products and norms on $H$ and $V$, respectively, by
    \begin{align}\label{def:H}
        \lp u, v\rp=\int_\Om u(x)\cdotp v(x)dx,\quad |u|^2=\lp u,u\rp,
    \end{align}
and
    \begin{align}\label{def:V}
        {\lpp u,v\rpp}=\sum_{j=1,2}\int_\Om \bdy_ju(x)\cdotp\bdy_jv(x)\ dx,\quad \lVert u\rVert^2={\lpp u, u\rpp}.
    \end{align}
Recall the Poincar\'e inequality, which implies the continuous embedding $V\subset H$:
    \begin{align}\label{est:Poincare}
        |u|\leq \lVert u\rVert.
    \end{align}
\begin{comment}
The dual spaces of $H, V$ will be denoted by $H^*, V^*$ respectively. Then we have the following continuous imbeddings
    \begin{align}\notag
        V\subset H\subset H^*\subset V^*.
    \end{align}
\end{comment}

For each $1\leq p\leq\infty$, we will also make use of the Lebesgue spaces, $L^p(\Om)$, which denote the space of $p$-integrable functions endowed with the following norm:
    \begin{align}\label{def:Lp}
        \Abs{p}{u}=\lp\int_\Om|u(x)|^p dx\rp^{1/p},
    \end{align}
with the usual modification when $p=\infty$. For convenience, we will view them as subspaces of absolutely integrable functions over $\Om$, which are mean-free and $2\pi$-periodic in each direction a.e. in $\Om$. It will be convenient to abuse notation and consider $L^p$ as a space of either scalar functions or vector fields. %In this way, we have $L^p\subset H\subset L^2$, for $p\geq2$.

\begin{comment}
%e denote the Stokes operator by $A=-\PP\De$ and define, 
We define for each integer $n\geq0$:
    \begin{align}\label{def:A}
        A^{n/2}u=\sum_{k\in\ZZ^2\smod\{(0,0)\}}\hat{u}_kw_k,\quad w_k(x)=\cos(k\cdotp x).
    \end{align}
The domain, $D(A^{n/2})$, of $A^{n/2}$ is a subspace of $H$ endowed with the topology induced by
    \begin{align}\label{def:Hn}
        \no{n}{u}=|A^{n/2}u|=\lp(2\pi)^2\sum_{k\in\ZZ^2}|k|^{2n}|\hat{u}_k|^2\rp^{1/2}.
    \end{align}
Observe that
    \begin{align}\notag
        |u|=\Abs{0}{u},\qquad 
        \rVert u\rVert=\no{0}{u}=\Abs{0}{A^{1/2}u}.
    \end{align}
\end{comment}

Our analysis will make use of the Ladyzhenska and Agmon interpolation inequalities: there exist absolute constants $C_L, C_A>0$ such that
	\begin{align}\label{est:interpolation}
		\Abs{4}{u}^2\leq C_L\lVert u\rVert|u|,\qquad
		\Abs{\infty}{u}^2\leq C_A|Au||u|.
	\end{align}
Another useful interpolation inequality is the following:
	\begin{align}\label{est:interpolation:CS}
		\lVert u\rVert^2\leq |Au||u|
	\end{align}
We will also make use of the Bernstein inequality: for any integers $m\leq n$
	\begin{align}\label{est:Bernstein}
		\Abs{}{A^{n/2}P_Nu}\leq N^{n-m}\Abs{}{A^{m/2}P_Nu},\quad \Abs{}{A^{m/2}Q_Nu}\leq N^{m-n}\Abs{}{A^{n/2}Q_Nu},
	\end{align}
where $A^{n/2}$ denotes powers of the Stokes operator, which is defined as
    \begin{align}\label{def:A}
        A^{n/2}u=\sum_{k\in\ZZ^2\smod\{(0,0)\}}|k|^n\hat{u}_kw_k,\quad w_k(x)=\exp(ik\cdotp x).
    \end{align}
\begin{comment}
Observe that we also have the following borderline Sobolev inequality
    \begin{align}\label{est:Sobolev}
        |P_Nu|_\infty\leq C_S(\ln N)^{1/2}\|P_Nu\|
    \end{align}
\end{comment}

Given $f\in L^\infty(0,\infty;H)$, the \textit{generalized Grashof number} is defined as
    \begin{align}\label{def:Grashof}
        \Gr:=\frac{\sup\limits_{t\geq0}|f(t)|}{\nu^2}.
    \end{align}
and its shape factor by
    \begin{align}\label{def:shape}
        \si_{-1}:=\frac{\sup\limits_{t\geq0}\|f(t)\|_*}{\sup\limits_{t\geq0}|f(t)|},
    \end{align}
where $\|\cdot\|_*$ denotes the norm on the space $V^*$ dual to $V$.

Upon recalling \eqref{eq:nse:ff} and \eqref{def:B}, we recall  the well-known, skew-symmetric property of the trilinear form $\lp B(u,v),w\rp$:
	\begin{align}\label{eq:B:skew}
		\lp B(u,v),w\rp=-\lp B(u,w),v\rp,
	\end{align}
for $u,v,w\in V$, which immediately implies
    \[
        \lp B(u,v),v\rp=0.    
    \]
%\begin{comment}
{We will also make use of the identity}
    \begin{align}\label{eq:B:enstrophy}
       {\lp B(u,u),Au\rp=0}, %{\lp B(v,v),Au\rp+\lp B(u,v),Av\rp+\lp B(v,u),Av\rp=0,}
    \end{align}
{which holds for $u\in V$, $v\in D(A)$ in the setting of mean-free, periodic functions over $[0,2\pi]^2$.}
%This immediately implies}
    %\[
    %    {\lp B(u,u),Au\rp=0.}
    %\]
%\end{comment}
Observe that $B:D(A)\times V\goesto H$ via 
    \begin{align}\label{est:B:ext:H}
        |B(u,v)|\leq C_A^{1/2}|Au|^{1/2}|u|^{1/2}\|v\|.
    \end{align}
Moreover, $B:V\times V\goesto V'$ is continuous and satisfies
    \begin{align}\label{est:B:ext}
        |\lp B(u,v),w\rp|\leq C_L\lVert u\rVert^{1/2}|u|^{1/2}\lVert v\rVert\|w\|^{1/2}|w|^{1/2}.
    \end{align} 
The Frech\'et derivative of $B$ will be denoted by $DB$. {{}It is straightforward to show that $DB$ is given by}
	\begin{align}\label{def:DB}
		DB(u)v=B(u,v)+B(v,u).
	\end{align}
By \eqref{est:B:ext:H}, it follows that $DB: D(A)\goesto L(D(A),H)$, $u\mapsto DB(u)$, while \eqref{est:B:ext} implies $DB: V\goesto L(V,V')$, where $L(X,Y)$ denotes the space of bounded linear operators $X\goesto Y$.

We recall the following classical global existence and uniqueness result for \eqref{eq:nse:ff}.

\begin{Prop}\label{prop:nse:ball}
Let $f\in L^\infty(0,\infty;H)$. Then for each $u_0\in V$ and $T>0$, there exists a unique solution $u\in C([0,T];V)\cap L^2(0,T;D(A))$ such that $u(0)=u_0$. Moreover, {for all $\de>0$}, there exists $t_0=t_0({\de},\|u_0\|,|f|)$ such that
    \begin{align}\label{est:absorb:L2}
        \sup\limits_{t\geq t_0}|u(t)|\leq {\sqrt{1+\de}}\nu\si_{-1}\Gr=:{\rho_0(\de)},\qquad \sup\limits_{t\geq t_0}\|u(t)\|\leq {\sqrt{1+\de}}\nu \Gr=:{\rho_1(\de)}.
    \end{align}
In fact, the balls $B_H({\rho_0(\de)})$ and $B_V({\rho_1(\de)})$ are forward-invariant sets for \eqref{eq:nse:ff} {for all $\de\geq0$}.
\end{Prop}

We will refer to the solutions guaranteed by \cref{prop:nse:ball} as \textit{strong solutions}. We note that the forward-invariance of $B_H({\rho_0(\de)})$ and $B_V({\rho_1(\de)})$ follow from the elementary inequalities which hold for strong solutions of \eqref{eq:nse:ff}:
    \begin{align}\label{est:abs:ball}
        \begin{split}
        |u(t)|^2&\leq e^{-\nu t}|u_0|^2+{\rho_0(\de)}^2(1-e^{-\nu t}),
        \\
        \lVert u(t)\rVert^2&\leq e^{-\nu t}\lVert u_0\rVert^2+{\rho_1(\de)}^2(1-e^{-\nu t}),
        \end{split}
    \end{align}
for all $t\geq0$ and $u_0\in V$.

We will also make use of the global well-posedness of the corresponding initial value problems for \eqref{eq:sync} and \eqref{eq:nudge}, {as well as their asymptotic synchronization properties}, which were developed in \cite{OlsonTiti2003} and \cite{AzouaniOlsonTiti2014}, respectively. We state them here for the sake of completeness. For both statements, given $f\in L_{loc}^\infty(0,\infty;H)$ and $u_0\in V$, we let $u$ denote the unique global-in-time solution to \eqref{eq:nse:ff} such that $u\in C([0,T];V)\cap L^2(0,T;D(A))$ and $\frac{du}{dt}\in L^2(0,T;H)$, for all $T>0$.

\begin{Prop}[Theorem 3.1 \& 3.3, \cite{OlsonTiti2003}]\label{thm:OT}
 For any $N>0$ and $q_0\in V$ such that $Q_Nq_0=q_0$, there exists a unique $q$ such that $q\in C([0,T];V)\cap L^2(0,T;D(A))$, $\frac{dq}{dt}\in L^2(0,T;H)$, for all $T>0$, and satisfies \eqref{eq:sync}. In particular, for $v=P_Nu+q$, the pair $(u,v)$ equivalently satisfies the following system of equations:
    \begin{align}\label{eq:sync:OT}
        \begin{split}
        \frac{du}{dt}+\nu Au+B(u,u)&=f,\quad u(0)=u_0\\
        \frac{dv}{dt}+\nu Av+B(v,v)&=f+P_N\left(B(v,v)-B(u,u)\right),\quad v(0)=P_Nu_0+q_0.
        \end{split}
    \end{align}
{Moreover, there exist positive constants $c_{OT}, K_{OT}$ such that for $N_{OT}=c_{OT}\rho_1\nu^{-1}$, it holds that}
    \begin{align}\label{eq:OT:converge}
        {|v(t;v_0)-u(t;u_0)|\leq K_{OT}|v_0-u_0|e^{-\frac{\nu}2t}},
    \end{align}
{for all $N\geq N_{OT}$}.
\end{Prop}

\begin{Prop}[Theorem 1 \& 6, \cite{AzouaniOlsonTiti2014}]\label{thm:AOT}
For any $N>0$ and $\tv_0\in V$, there exists a unique $\tv$ such that $\tv\in C([0,T];V)\cap L^2(0,T;D(A))$, $\frac{d\tv}{dt}\in L^2(0,T;H)$, for all $T>0$, and satisfies \eqref{eq:nudge}. In particular, the pair $(u,\tv)$ satisfies the following system of equations:
    \begin{align}\label{eq:nudge:AOT}
        \begin{split}
        \frac{du}{dt}+\nu Au+B(u,u)&=f,\quad u(0)=u_0\\
        \frac{d\tv}{dt}+\nu A\tv+B(\tv,\tv)&=f-\mu P_N\tv+\mu P_Nu,\quad \tv(0)=\tv_0.
        \end{split}
    \end{align}
{Moreover, there exist positive constants $c_{AOT}, K_{AOT}$ such that for $N_{AOT}=c_{AOT}\rho_1\nu^{-1}$ and $\frac{1}4\nu N_{AOT}^2\leq \mu\leq \frac{1}4\nu N^2$, it holds that}
    \begin{align}\label{eq:AOT:converge}
        {|\tv(t;\tv_0)-u(t;u_0)|\leq K_{AOT}|\tv_0-u_0|e^{-(\mu/2)t},}
    \end{align}
{for all $N\geq N_{AOT}$.}
\end{Prop}

\section{The Infinite Nudging Limit of the {Nudging Algorithm}}\label{sect:limit}

Throughout this section, we let $f\in L^\infty(0,\infty;H)$. For $u_0\in V$, let $u$ denote the unique global-in-time strong solution of \eqref{eq:nse:ff} corresponding to $u_0$ guaranteed by \cref{prop:nse:ball}. Without loss of generality, we will assume throughout this section that the reference solution has evolved sufficiently far in time to satisfy the estimates \eqref{est:absorb:L2} at $t=0$. In particular, we may suppose that $t_0=0$ in \cref{prop:nse:ball}. 

%Note that this is equivalent to assuming that all initial conditions belong to $B_H(\rho_0)\cap B_V(\rho_1)$.

%The main result of this section is given by the following statement.

Given $N>0$, let $p_0=P_Nu_0$, $p=P_Nu$ and ${{}Q}=Q_Nu$, so that ${{}Q(0)=Q_Nu_0={{}Qu_0}}$ and $u=p+{{}Q}$. Then {the true dynamics, as determined by \eqref{eq:nse:ff}, are {{}equivalently} represented by}
    \begin{align}\label{eq:nse:high:low}
        \begin{split}
        \frac{dp}{dt}+\nu Ap+P_NB(p+{{}Q},p+{{}Q})&=P_Nf,\quad p(0)=p_0,
        \\
        \frac{d{{}Q}}{dt}+\nu A{{}Q}+Q_NB(p+{{}Q},p+{{}Q})&=Q_Nf,\quad {{}Q}(0)={{}{{}Qu_0}}
        \end{split}
    \end{align}
Now, given $q_0\in Q_NV$, we let $v$ denote the unique output of the {direct-replacement algorithm} \eqref{def:sync} guaranteed by \cref{thm:OT}, so that $P_Nv=p$. Then if we denote $q=Q_Nv$, it follows that $v=p+q$, where $p,q$ satisfy
    \begin{align}\label{eq:nse:sync:high:low}
        \begin{split}
        \frac{dp}{dt}+\nu Ap+P_NB(p+{{}Q},p+{{}Q})&=P_Nf,\quad p(0)=p_0,
        \\
        \frac{dq}{dt}+\nu Aq+Q_NB(p+q,p+q)&=Q_Nf,\quad q(0)=q_0.
        \end{split}
    \end{align}
Lastly, given $\tv_0\in V$, we let $\tv$ denote the unique, global-in-time strong solution of the {{}nudging algorithm} \eqref{eq:nudge} corresponding to $\tv_0$ guaranteed by \cref{thm:AOT}. We let $\tp=P_N\tv$ and $\tq=P_N\tv$, so that $\tv_0=\tp_0+\tq_0$. Then
    \begin{align}\label{eq:nse:nudge:high:low}
        \begin{split}
        \frac{d\tp}{dt}+\nu A\tp+P_NB(\tp+\tq,\tp+\tq)&=P_Nf-\mu \tp+\mu p,\quad \tp(0)=\tp_0,
        \\
        \frac{d\tq}{dt}+\nu A\tq+Q_NB(\tp+\tq,\tp+\tq)&=Q_Nf,\quad \tq(0)=\tq_0.
        \end{split}
    \end{align}
{{}With these three representations \eqref{eq:nse:high:low}, \eqref{eq:nse:sync:high:low}, \eqref{eq:nse:nudge:high:low} now in view, upon recalling the notation developed in \cref{sect:intro}, it is now clear} that $\tsS(p;\tv_0)=\tp+\sS(\tp;
\tq_0)$. {{}This representation of $\tsS$ allows us to study the high-modes of the nudging variable, $\tv$, through the mapping $\sS$, and thus allow for a convenient comparison with the high-modes of the direct-replacement variable, $v$.} {For the convenience of the reader, let us quickly summarize the notation above, which we will maintain for the remainder of the paper:}
    \[
        {u=p+{{}Q}=P_Nu+Q_Nu,\quad v=p+q=P_Nv+Q_Nv,\quad \tv=\tp+\tq=P_N\tv+Q_N\tv,}
    \]
{with corresponding initial values given by}
    \[
        {u(0)=u_0=p_0+{{}{{}Qu_0}},\quad v(0)=v_0=p_0+q_0,\quad \tv(0)=\tv_0=\tp_0+\tq_0.}
    \]
{{}We emphasize that our convention is such that $P_Nv_0=P_Nu_0=p_0$ always holds, while this is not necessarily imposed for $\tp_0=P_N\tv_0$; we will treat the cases when $\tp_0=p_0$ and $\tp_0\neq p_0$ separately.} 

{Ultimately, we prove \cref{thm:main} in steps. First, we establish a stronger version of \eqref{eq:main:claim2}, where we allow the time-horizon to grow as $\mu$ increases, and we also accommodate the case of incompatible initial data, i.e., $\tp_0\neq p_0$, but $\tq_0=q_0$, between the nudging and direct-replacement algorithm.} 

\begin{Thm}\label{thm:nudge:limit}
{Let $v_0,\tv_0\in V$}. Then there exists $T:[0,\infty)\goesto[0,\infty)$, $\mu\mapsto T(\mu)$, such that $T$ is strictly increasing, $\lim\limits_{\mu\goesto\infty}T(\mu)=\infty$, and {{}if $\tv_0=v_0$}, then
    \begin{align}\label{est:nudge:infinity}
        \lim\limits_{\mu\goesto\infty}\sup\limits_{t\in[{0},T(\mu)]}|\tv(t;v_0)-v(t;v_0)|=0,
    \end{align}
{Moreover, if $\tp_0\neq p_0$, but $\tq_0=q_0$, then for all $T>0$
    \begin{align}\label{est:nudge:infinity:positive}
        {{}\lim\limits_{t_0\goesto0^+}}\limsup\limits_{\mu\goesto\infty}\sup\limits_{t\in[t_0,T]}|\tv(t;{{}\tp_0+q_0})-v(t;{{}p_0+q_0})|=0.
    \end{align}
}
\end{Thm}

{Next, we establish a variation on the convergence result of Azouani-Olson-Titi \cite[Theorem 1]{AzouaniOlsonTiti2014}, which allows us guarantee asymptotic synchronization of the {nudging algorithm} with the true solution provided that enough modes are observed.}

\begin{Prop}\label{thm:nudge:converge}
{
%Let $u$ be a strong solution of \eqref{eq:nse:ff} corresponding to initial data $u_0\in V$. 
For $N_0=4C_L\rho_1\nu^{-1}$, it holds that
    \begin{align}\label{est:nudge:converge}
        |\tv(t;\tv_0)-u(t;u_0)|\leq e^{-\frac{\nu}2t}|\tv_0-u_0|.
    \end{align}
for all $N\geq N_0$, $\mu\geq \frac{9}{16}N_0^2\nu$, and $\tv_0\in V$.
}
\end{Prop}

{Note that in contrast with \cref{thm:AOT}, \textit{no upper bound is imposed} on $\mu$ in \cref{thm:nudge:converge}. However, the effective trade-off in doing so is in the exponential rate of convergence, sacrificing a rate $\mu$ in \eqref{eq:AOT:converge} for the potentially much slower rate of $\nu$ in \eqref{est:nudge:converge}. Upon combining \cref{thm:nudge:limit} and \cref{thm:nudge:converge}, we may immediately obtain \eqref{eq:main:claim3} as a corollary. Note that we also analogously establish the case of incompatible initial data.}

\begin{Cor}\label{cor:nudge:complete}
{There exists $N_*$ such that if $\tv_0=v_0$, then for all $N\geq {{}N_*}$
    \begin{align}\label{eq:nudge:complete}
        \lim\limits_{\mu\goesto\infty}\sup\limits_{t\geq 0}|\tv(t;{{}v_0})-v(t;v_0)|=0,
    \end{align}
Moreover, if $\tp_0\neq p_0$, but $\tq_0=q_0$, then
    \begin{align}\label{eq:nudge:complete:positive}
        {\lim\limits_{t_0\goesto0^+}\limsup\limits_{\mu\goesto\infty}\sup\limits_{t\geq t_0}|\tv(t;{{}\tp_0+q_0})-v(t;{{}p_0+q_0})|=0.}
    \end{align}
}
\end{Cor}

{{{}We point out that the key insight in establishing convergence of the nudging algorithm to the direct replacement algorithm over an infinite time horizon is that when $N$ is sufficiently large, the two algorithms already asymptotically converge to the true dynamics that is commonly being observed. In this regime, convergence in the infinite nudging limit over all time reduces to establishing convergence in the infinite nudging limit over \textit{arbitrary finite time intervals}; of course this property is precisely what is guaranteed by \cref{thm:nudge:limit}. Indeed, assuming  \cref{thm:nudge:limit}, \cref{thm:nudge:converge} are true, let us first prove \cref{cor:nudge:complete}.}}

{
\begin{proof}[Proof of \cref{cor:nudge:complete}]

For any $T>0$, observe that
    \begin{align}
        &\sup\limits_{t\geq 0}|\tv(t;\tv_0)-v(t;v_0)|\notag
        \\
        &\leq \sup\limits_{t\in[0,T]}|\tv(t;\tv_0)-v(t;v_0)|+\sup\limits_{t\geq T}|\tv(t;\tv_0)-u(t;u_0)|+\sup\limits_{t\geq T}|u(t;u_0)-v(t;v_0)|.\notag
    \end{align}
Fix $\veps>0$. Let $N_*=\max\{N_{OT}, N_0\}$, $\mu \geq N_*^2\nu$, and $N\geq N_*$. By \cref{thm:nudge:converge}, there exists $T_0^{(\veps)}$ such that for all $T\geq T_0^{(\veps)}$
    \begin{align}\label{est:aot}
        \sup\limits_{t\geq T}|\tv(t;v_0)-u(t;u_0)|<\frac{\veps}3.
    \end{align}
By \cref{thm:OT}, there exists $T_{OT}^{(\veps)}$ such that for all $T\geq T_{OT}^{(\veps)}$
    \begin{align}\label{est:ot}
        \sup\limits_{t\geq T}|u(t;u_0)-v(t;v_0)|<\frac{\veps}3.
    \end{align}
Let $T_\veps:=\max\{T_0^{(\veps)}, T_{OT}^{(\veps)}\}$. 

If $\tv_0=v_0$, then by \cref{thm:nudge:limit} \eqref{est:nudge:infinity}, there exists $\mu_\veps$ such that for all $\mu\geq \mu_\veps$
    \begin{align}\notag
        \sup\limits_{t\in[0,T_\veps]}|\tv(t;v_0)-v(t;v_0)|<\frac{\veps}3.
    \end{align}
Upon combining this with {{}\eqref{est:aot} and \eqref{est:ot}}, we therefore conclude that
    \[
    \sup\limits_{t\geq {}0}|\tv(t;v_0)-v(t;v_0)|<\veps,
    \]
holds for all $\mu \geq \max\{\mu_\veps, N_*^2\nu\}$, which establishes \eqref{eq:nudge:complete}.

On the other hand, {{}fix $t_0\in(0,T]$. If} $\tq_0=q_0$, but $\tp_0\neq p_0$, then by \cref{thm:nudge:limit} \eqref{est:nudge:infinity:positive}, there exists $\mu_\veps$ such that for all $\mu\geq \mu_\veps$ 
    \[  
        \sup\limits_{t\in[{{}t_0},T]}|\tv(t;\tv_0)-{{}v}(t;u_0)|\leq {{}o(1)},
    \]
{{}as a sequence in $t_0$. \cref{thm:nudge:limit} \eqref{est:nudge:infinity:positive} then further asserts that there exists $t_\veps\in(0,T]$ such that for all $t_0\in(0,t_\veps)$, one has
        \[  
        \sup\limits_{t\in[{{}t_0},T]}|\tv(t;\tv_0)-{{}v}(t;u_0)|\leq \frac{\veps}3,
    \]
} Upon combining this with {{}\eqref{est:aot} and \eqref{est:ot} for $T\geq T_\veps$}, we deduce \eqref{eq:nudge:complete:positive}.
\end{proof}
}

{{{}Therefore, \cref{thm:main} will be proved upon establishing both \cref{thm:nudge:limit} and \cref{thm:nudge:converge}}. The remainder of this section is now dedicated to proving \cref{thm:nudge:limit} and \cref{thm:nudge:converge}. To simplify the forthcoming analysis, note that by \eqref{est:abs:ball}, it will be convenient to assume that $u_0\in B_H(\rho_0(\de))$ and $B_V(\rho_1(\de))$. Indeed, given $u_0\in V$, there exists $\de>0$ such that $|u_0|\leq \rho_0(\de)$ and $\|u_0\|\leq\rho_1(\de)$. Henceforth, it is without loss of generality that we set the convention that $\rho_0=\rho_0(\de)$ and $\rho_1=\rho_1(\de)$.} 

We first seek to prove  \cref{thm:nudge:limit}. To this end, we begin by establishing an elementary stability estimate.

\begin{Lem}\label{lem:w:bad}
Let $u_0,\tv_0\in V$. Suppose that $u_0\in {B_V(\rho_1)}$. Then
    \begin{align}\label{est:w:bad}
        \sup\limits_{t\in[0,T]}{\|\tv(t;\tv_0)-u(t;u_0)\|^2}\leq {C_*(T)^2\|\tv_0-u_0\|^2},
    \end{align}
for all $T>0$ and $N>0$, where 
    \begin{align}\label{def:Cstar}
        C_*(T)=\exp\left[{\frac{27C_A^2}{256}}\left(\frac{\rho_1}{\nu}\right)^2\nu T\right].
    \end{align}
\end{Lem}

\begin{proof}
Let $w=\tv-u$. Then
    \begin{align}\label{eq:w}
        \frac{dw}{dt}+\nu Aw+B(w,w)+DB(u)w=-\mu \tp+\mu p,\quad w(0)={\tv_0-u_0}.
    \end{align}
{Observe that $w(0)=(\tp_0-p_0)+(\tq_0-q_0)$}. Upon taking the {$V$}--inner product of \eqref{eq:w} with $w$ and invoking \eqref{eq:B:enstrophy}, we obtain
    \begin{align}\notag
        {\frac{1}2\frac{d}{dt}\|w\|^2+\nu|Aw|^2+\mu\|P_Nw\|^2=-\lpp B(w,u),w\rpp}.
    \end{align}
By {H\"older's inequality}, \eqref{est:B:ext:H}, and Young's inequality, we have
    \begin{align}\notag
        |{\lpp B(w,u),w\rpp}|&\leq {C_A^{1/2}|Aw|^{3/2}|w|^{1/2}\|u\|\leq \nu|Aw|^2+\frac{27C_A^2}{256\nu^3}\|u\|^4\|w\|^2}.
    \end{align}
Thus, by \cref{prop:nse:ball} and \eqref{est:abs:ball}, we have
    \begin{align}\notag
        {\frac{d}{dt}\|w\|^2+\mu\|P_Nw\|^2}\leq \nu {\frac{27C_A^2}{128}\lp\frac{\rho_1}{\nu}\rp^4\|w\|^2}.
    \end{align}
By Gr\"onwall's inequality, we therefore deduce 
    \begin{align}\notag
        {\|w(t)\|^2\leq \exp\left(\frac{27C_A^2}{128}\lp\frac{\rho_1}{\nu}\rp^2 \nu t\right)\|\tv_0-u_0\|^2},
    \end{align}
which implies \eqref{est:w:bad}.
\end{proof}

Next, we show how the stability estimate \cref{lem:w:bad} yields a stability estimate on the low-mode error with a favorable dependence on $\mu$.

\begin{Lem}\label{lem:y}
Let $u_0,\tv_0\in V$. Suppose that $u_0\in B_H(\rho_0)\cap B_V(\rho_1)$. Then
    \begin{align}\label{est:y}
        |\tp(t;\tp_0)-p(t;p_0)|^2\leq e^{-2\mu t}|\tp_0-p_0|^2+ \frac{\nu^3}{\mu}\tC({\tv_0,u_0},T)^2,
    \end{align}
for all $0\leq t\leq T$, $T>0$, and $N>0$, where
    \begin{align}\label{def:tC}
        \tC({\tv_0,u_0},T)^2=2C_L^2\left[\left(\frac{\rho_0\rho_1}{\nu^2}\right)^2+C_*(T)^4{\left(\frac{\|\tv_0-u_0\|^2}{\nu^2}\right)^2}\right],
    \end{align}
where $C_*(T)$ is given by \eqref{def:Cstar}. In particular, {for any $t_0>0$, one has}
    \begin{align}\label{est:y:transient}
        {\sup\limits_{t\geq t_0}}|\tp(t;\tp_0)-p(t;p_0)|^2\leq {e^{-2\mu t_0}}|\tp_0-p_0|^2+ \frac{\nu^3}{\mu}\tC(\tv_0,u_0,T)^2,
    \end{align}
{while, on the other hand}, if $\tp_0=p_0$, then
    \begin{align}\label{est:y:alt}
        \sup\limits_{t\in[0,T]}|\tp(t;p_0)-p(t;p_0)|^2\leq \frac{\nu^3}{\mu}\tC({\tq_0},{{}{{}Qu_0}},T)^2.
    \end{align}
\end{Lem}

\begin{proof}
Let $y=\tp-p$. Then $y_0=\tp_0-p_0$ and
	\begin{align}
		\frac{dy}{dt}+\nu Ay+P_NB(\tv,\tv)-P_NB(u,u)&=-\mu y,\quad y(0)=y_0.\label{eq:error:low}
	\end{align}
In particular, for $w=\tv-u$, \eqref{eq:error:low} can be rewritten as
    \begin{align}\label{eq:y}
		\frac{dy}{dt}+\nu Ay&=-P_NB(w,w)-DP_NB(u)w-\mu y.
    \end{align}
Upon taking the $H$--inner product of \eqref{eq:y} with $y$, one obtains the following energy balance for the low-mode error:
	\begin{align}\label{eq:y:balance}
	\frac{1}2\frac{d}{dt}|y|^2+\nu\|y\|^2+\mu|y|^2&=-\lp B(w,w)+DB(u)w,y\rp.
	\end{align}
By \eqref{eq:B:skew}, \eqref{est:B:ext}, and Young's inequality, we have
    \begin{align}
        |\lp B(w,w),y\rp|&=|\lp B(w,y),w\rp|\notag
        \\
        &\leq C_L\|w\|\|y\||w|\leq \frac{\nu}2\|y\|^2+\frac{C_L^2}{2\nu}\|w\|^2|w|^2\notag
        \\
        |\lp DB(u)w,y\rp|&\leq |\lp B(u,y),w\rp|+|\lp B(w,y),u\rp|\notag
        \\
        &\leq 2C_L\|u\|^{1/2}|u|^{1/2}\|y\|\|w\|^{1/2}|w|^{1/2}\notag
        \\
        &\leq \frac{\nu}2\|y\|^2+C_L^2\nu^3\left[\left(\frac{\|u\||u|}{\nu^2}\right)^2+\left(\frac{\|w\||w|}{\nu^2}\right)^2\right].\notag
    \end{align}
Combining these estimates in \eqref{eq:y:balance} yields
    \begin{align}\notag
        \frac{d}{dt}|y|^2+2\mu|y|^2&\leq 4C_L^2\nu^3\left[\left(\frac{\|u\||u|}{\nu^2}\right)^2+\left(\frac{\|w\||w|}{\nu^2}\right)^2\right].
    \end{align}
Applying \cref{prop:nse:ball}, \eqref{est:abs:ball}, and \cref{lem:w:bad} gives
    \begin{align}
        \frac{d}{dt}|y|^2+2\mu|y|^2&\leq 4C_L^2\nu^3\left[\left(\frac{\rho_0\rho_1}{\nu^2}\right)^2+C_*(T)^4{\left(\frac{\|\tv_0-u_0\|^2}{\nu^2}\right)^2}\right].\notag
    \end{align}
It then follows from Gr\"onwall's inequality {and orthogonality} that
    \begin{align}\notag
        |y(t)|^2\leq e^{-2\mu t}|y_0|^2+ 2C_L^2\frac{\nu^3}{\mu}\left\{\left(\frac{\rho_0\rho_1}{\nu^2}\right)^2+C_*(T)^4{\left[\left(\frac{\|\tp_0-p_0\|}{\nu}\right)^2+\left(\frac{\|\tq_0-{{}{{}Qu_0}}\|}{\nu}\right)^2\right]^2}\right\}.
    \end{align}
{This establishes \eqref{est:y}, and when $\tp_0=p_0$, it establishes \eqref{est:y:alt}, as desired.}
\end{proof}

The last ingredient is to show that the operator $\sS$ mapping $p\mapsto \sS(p)=q$, where $q$ satisfies the high-mode component of \eqref{eq:nse:sync:high:low}, is a local Lipschitz mapping. In order to prove this, we will require {a priori} bounds on \eqref{eq:sync}. Let us therefore establish these {a priori} bounds first.
\begin{comment}
To this end, it will be convenient to generalize \eqref{eq:sync} by assuming that $p(\cdotp)$ is an arbitrarily given continuous trajectory contained in $H$. In particular, $p(\cdotp)$ need not represent the low-modes of a solution to the Navier-Stokes equation. 
Let $0<T\leq \infty$, $N>0$, and $q_0\in H$. We consider the map $\sS(\cdotp;q_0):C([0,T);P_NH)\goesto C([0,T);H)$, $p\mapsto \sS(p;q_0)=q$, such that
    \begin{align}\label{eq:sync}
        \frac{dq}{dt}+\nu Aq+Q_NB(p+q,p+q)=Q_Nf,\quad q(0)=q_0,\quad t\in(0,T].
    \end{align}
In particular, $\sS$ denotes the solution operator of the initial value problem \eqref{eq:sync:app}. Our current goal is to develop an {a priori} estimates for solutions of \eqref{eq:sync:app} for each given $q_0\in H$, that is, we obtain a suitable upper bound on the solution operator $\sS(\cdotp;q_0)$. 
\end{comment}
%In what follows, we let
    %\begin{align}\label{def:sync:app}
     %   v:=p+\sS(p;q_0),\quad q=\sS(p;q_0),
    %\end{align}
{In particular, we} claim the following.

\begin{Lem}\label{lem:sync}
For all $T\in(0,\infty]$, $N>0$, and $q_0\in {V}$, there exists $\Cq>0$, {{}independent of $\mu$}, such that
    \begin{align}\label{est:sync}
         \sup\limits_{t\in[0,T]}\|\sS(p;q_0)(t)\|\leq\Cq,
    \end{align}
In particular
    \begin{align}\label{def:Cq}
        \Cq^2=\exp\left[4C_AN^2\left(\frac{\Cp}{\nu}\right)^2\nu T\right]\left\{\|q_0\|^2+\frac{\nu^{2}}2\left[N^2\left(\frac{\Cp}{\nu}\right)^{2}+\frac{\Gr^2}{C_AN^2}\left(\frac{\Cp}{\nu}\right)^{-2}\right]\right\},
    \end{align}
{where we let $\Cp$ denote the constant}
\begin{align}\label{est:p}
        {\Cp=\sup\limits_{t\in[0,T]}|p(t)|.}
    \end{align}
\end{Lem}

{{}We emphasize that in the subsequent application of \cref{lem:sync}, the constant $\Cp$ will also be independent of $\mu$; the independence of $\Cp,\Cq$ on $\mu$ is crucial.}

\begin{proof}
The enstrophy balance for \eqref{eq:sync} is given by
    \begin{align}\notag
        \frac{1}2\frac{d}{dt}\|q\|^2+\nu|Aq|^2=-\lp B(p,p)+B(q,p)+B(p,q),Aq\rp+\lp f,Aq\rp.
    \end{align}
By \eqref{eq:B:skew}, \eqref{eq:B:enstrophy}, \eqref{est:B:ext}, \eqref{est:interpolation}, \eqref{est:Poincare}, \eqref{est:Bernstein}, \eqref{est:p}, and Young's inequality we have
    \begin{align}
        |\lp B(p,p), Aq\rp|&\leq C_A^{1/2}|Ap|^{1/2}|p|^{1/2}\|p\||Aq|\leq C_A^{1/2}N^2\Cp^2|Aq|\leq \frac{C_AN^4}{\nu}\Cp^4+\frac{\nu}4|Aq|^2\notag
        \\
        |\lp B(q,p),Aq\rp|&\leq |q||\nabla p|_\infty|Aq|\leq \frac{C_A^{1/2}}N\|A p\|^{1/2}\|p\|^{1/2}\|q\||Aq|\leq C_A^{1/2}N\Cp\|q\||Aq|\notag
        \\
        &\leq \frac{C_AN^2}{\nu}\Cp^2\|q\|^2+\frac{\nu}4|Aq|^2\notag
        \\
        |\lp B(p,q),Aq\rp|&\leq C_A^{1/2}|Ap|^{1/2}|p|^{1/2}\|q\||Aq|\leq C_A^{1/2}N|p|\|q\||Aq|\leq \frac{C_AN^2}{\nu}\Cp^2\|q\|^2+\frac{\nu}4|Aq|^2.\notag
    \end{align}
Also, by the Cauchy-Schwarz inequality, Young's inequality, and \eqref{def:Grashof}  we have
    \begin{align}
        |\lp f,Aq\rp|&\leq |f||Aq|\leq \nu^3\Gr^2+\frac{\nu}4|Aq|^2.\notag
    \end{align}
Upon combining the above, we arrive at
    \begin{align}
        \frac{d}{dt}\|q\|^2\leq 4C_AN^2\nu\left(\frac{\Cp}{\nu}\right)^2\|q\|^2+2C_AN^4\nu^3\left(\frac{\Cp}{\nu}\right)^4+2\nu^3\Gr^2\notag.
    \end{align}
An application of Gr\"onwall's inequality, then yields
    \begin{align}
        \|q(t)\|^2\leq \exp\left[4C_AN^2\left(\frac{\Cp}{\nu}\right)^2\nu T\right]\left\{\|q_0\|^2+\frac{\nu^{2}}2\left[N^2\left(\frac{\Cp}{\nu}\right)^{2}+\frac{\Gr^2}{C_AN^2}\left(\frac{\Cp}{\nu}\right)^{-2}\right]\right\},\notag
    \end{align}
as desired.
\end{proof}

\begin{comment}
From \cref{lem:sync:app}, which asserts that for each $T>0$, $N>0$, and $q_0\in H$, and $p\in C([0,T];H)$, there exists $\Cq<\infty$ such that
    \begin{align}\label{est:sync}
        \sup\limits_{t\in[0,T]}\|q(t)\|\leq \Cq,
    \end{align}
where, for $\Cp=\sup\limits_{t\in[0,T]}|p(t)|$, $\Cq$ is given by
    \begin{align}\label{def:Cq}
         \Cq^2=\exp\left[4C_AN^2\left(\frac{\Cp}{\nu}\right)^2\nu T\right]\left\{\|q_0\|^2+\frac{\nu}2\left[N^2\left(\frac{\Cp}{\nu}\right)^4+\frac{\Gr^2}{C_AN^2}\right]\right\}
    \end{align}
\end{comment}

We are now ready to establish the local Lipschitz property of the operator $\sS(\cdotp;q_0)$.

\begin{Thm}\label{thm:S:lipschitz}
For each $T>0$, $N>0$, $q_0\in V$, the map $\sS(\cdotp;q_0):C([0,T];P_NH)\goesto C([0,T];{{}Q_NH})$ is locally Lipschitz. In particular, for any ball $B({\Cp_0})\subset C([0,T];P_NH)$ of radius ${\Cp_0}>0$, {centered at $0$}, there exists a constant $C_\sS$ such that 
    \begin{align}\label{est:S:lipschitz}
        \sup\limits_{t\in[0,T]}|\sS(p_1;q_0)(t)-\sS(p_2;q_0)(t)|\leq C_\sS{\left(\nu T\right)^{1/2}}\sup\limits_{t\in[0,T]}|p_1(t)-p_2(t)|,
    \end{align}
whenever $p_1,p_2\in B({\Cp_0})$, {where the constant $C_\sS$ is given by}
    \begin{align}\label{def:Cs}
        &{C_{\sS}^2\leq
       {2N^2}\left[{{C_L^2}\left(\frac{2{\Cp_0}}{\nu}\right)^2}+{4}C_A\right]\exp\left\{{2}C_L^2N^2\left[\left(\frac{{{\Cp_0}}+{\Cq_*}}{\nu}\right)^2+{\left(\frac{2{\Cp_0}}{\nu}\right)^2}\right]\nu T\right\}},
        %{\left(1+\frac{4C_A}{C_L^2}\right)\exp\left\{2C_L^2N^2\left[\left(\frac{{\Cp_0}+\Cq}{\nu}\right)^2+\left(\frac{2{\Cp_0}}{\nu}\right)^2\right]\nu T\right\}}.
    \end{align}
{where $\Cq_*:=\min\left\{\sup\limits_{t\in[0,T]}|q_1(t)|,\sup\limits_{t\in[0,T]}|q_2(t)|\right\}$.} {Moreover, given $\tq_0\in V$, we have}
    \begin{align}\label{est:S:lipschitz:positive}
        \text{\small $\sup\limits_{t\in[0,T]}|\sS(p_1;{\tq_0})(t)-\sS(p_2;q_0)(t)|\leq {\tilde{C}_{\sS,1}|\tq_0-q_0|}+{\tilde{C}_{\sS,2}}\sup\limits_{t\in[0,T]}|p_1(t)-p_2(t)|,$}
    \end{align}
{where}
    \begin{align}\label{def:Cstilde}
        \begin{split}
        &\tilde{C}_{\sS,1}^2\leq {\exp\left\{2C_L^2N^2\left[\left(\frac{{\Cp_0}+{\Cq_*}}{\nu}\right)^2+{\left(\frac{\sup\limits_{t\in[0,T]}|p_1(t)-p_2(t)|}{\nu}\right)^2}\right]\nu T\right\}}
        \\
        &\tilde{C}_{\sS,2}^2\leq \left({{}1}+\frac{{4}C_A}{C_L^2}\right){\tilde{C}_{\sS,1}^2}.
        \end{split}
    \end{align}
\end{Thm}

\begin{proof}
For $j=1,2$, let $p_j\in C([0,T]; P_NH)$, $q_j=\sS(p_j(\cdotp);q_0)$, and $v_j=p_j+q_j$. Let $\pi=p_1-p_2$ and $\kap=q_1-q_2$, so that $\kap(0)=0$. Then
    \begin{align}\label{eq:kap}
        \frac{d\kap}{dt}+\nu A\kap+Q_NB(p_1+q_1,p_1+q_1)-Q_NB(p_2+q_2,p_2+q_2)=0,\quad \kap(0)=0.
    \end{align}
For $j=1,2$, let $\Cp_j=\sup\limits_{t\in[0,T]}|p_j(t)|$ and $\Cq_j=\sup\limits_{t\in[0,T]}\|{{}q_j(t)}\|$, {so that $\Cp_*=\min\{\Cp_1,\Cp_2\}$ and $\Cq_*=\min\{\Cq_1,\Cq_2\}$.} {Note that $\Cp_*<\Cp$ by assumption.}
%, where $\Cq_j$ is the constant in \eqref{est:sync}. 

Now observe that
    \begin{align}\label{eq:B:rewrite}
        &B(p_1+q_1,p_1+q_1)-B(p_2+q_2,p_2+q_2)\notag
        \\
        &=B(p_1,p_1)+B(p_1,q_1)+B(q_1,p_1)+B(q_1,q_1)\notag
        \\
        &\quad-\left(B(p_2,p_2)+B(p_2,q_2)+B(q_2,p_2)+B(q_2,q_2)\right)\notag
        \\
        &=B(\pi,\pi)+DB({v_2})\pi+B(\kap,\kap)+DB(v_2)\kap+DB(\pi)\kap.
        \end{align}
Then the energy balance for \eqref{eq:kap} is given by
    \begin{align}\notag
        \frac{1}2\frac{d}{dt}|\kap|^2+\nu\|\kap\|^2=-\lp B(\kap,v_2)+B(\kap,\pi),\kap\rp-\lp B(\pi,\pi)+DB({{}v_2})\pi,\kap\rp.
    \end{align}
By \eqref{eq:B:skew}, \eqref{est:B:ext}, \eqref{est:interpolation}, \eqref{est:Bernstein}, \eqref{est:sync}, and Young's inequality we have
    \begin{align}
        |\lp B(\kap,v_2),\kap\rp|&\leq C_L\|v_2\|\|\kap\||\kap|%\leq \textcolor{red}{C_L^2\nu\left(\frac{N\Cp_1+\Cp_2}{\nu}\right)^2|\kap|^2+\frac{\nu}4\|\kap\|^2}\notag
        {\leq 
        C_L^2 \nu \lp \frac{N\Cp_2 + \Cq_2}{\nu}\rp^2|\kappa|^2 + \frac{\nu}4\|\kappa\|^2
        }\notag
        \\
        %|\lp B(\kap,\pi),\kap\rp|&\leq C_L\|\pi\|\|\kap\||\kap|\leq C_LN|\pi|\|\kap\||\kap|\leq C_L^2N^2\nu\left(\frac{\Cp_1+\Cp_2}{\nu}\right)^2|\kap|^2+\frac{\nu}4\|\kap\|^2\notag
        |\lp B(\kap,\pi),\kap\rp|&\leq C_L\|\pi\|\|\kap\||\kap|\leq C_LN|\pi|\|\kap\||\kap|\leq {\frac{C_L^2N^2}{\nu}|\pi|^2}|\kap|^2+\frac{\nu}4\|\kap\|^2\notag
        \\
        |\lp B(\pi,\pi),\kap\rp|&\leq C_L\|\pi\||\pi|\|\kap\|\leq C_LN|\pi|^2\|\kap\|
        %\leq \textcolor{red}{C_L^2N^2\nu\left(\frac{\Cp_1+\Cp_2}{\nu}\right)^2|\pi|^2}\notag
        %\\
        %{\leq C_L^2N^2\nu\left(\frac{\Cp_1+\Cp_2}{\nu}\right)^2|\pi|^2 + \frac{\nu}4\|\kappa\|^2}\notag
        \leq {\frac{C_L^2N^2}{\nu}|\pi|^4} + \frac{\nu}4\|\kappa\|^2\notag
        \\
        |\lp DB({v_2})\pi,\kap\rp|&\leq 2C_A^{1/2}|{v_2}|\|\kap\||{A}\pi|^{1/2}|\pi|^{1/2}\leq 2C_A^{1/2}N|{v_2}|\|\kap\||\pi|\notag
        \\
        &\leq 4C_AN^2\nu\left(\frac{\Cp_2+{\Cq_2}}{\nu}\right)^2|\pi|^2+\frac{\nu}4\|\kap\|^2\notag.
    \end{align}
Upon combining the above estimates in \eqref{eq:kap}, we arrive at
    \begin{align}
        \frac{d}{dt}|\kap|^2%\textcolor{red}{\leq 2C_L^2N^2\nu\left[\left(\frac{\Cp_2+\Cq_2}{\nu}\right)^2+\left(\frac{\Cp_1+\Cp_2}{\nu}\right)^2\right]|\kap|^2}\notag
        %\\
        %&\textcolor{red}{\quad+2N^2\nu\left[C_L^2\left(\frac{\Cp_1+\Cp_2}{\nu}\right)^2+2C_A\left(\frac{\Cp_2}{\nu}\right)^2\right]|\pi|^2.}\notag
        %\\
        &{\leq 
        2C_L^2N^2 \nu \left[\lp \frac{\Cp_2 + \Cq_2}{\nu}\rp^2 +
        %\left(\frac{\Cp_1+\Cp_2}{\nu}\right)^2
        {\left(\frac{|\pi|}{\nu}\right)^2}\right]|\kap|^2}\notag
        \\
        &{\quad+ 2N^2\nu\left[C_L^2{\left(\frac{|\pi|}{\nu}\right)^2}
        %\left(\frac{\Cp_1+\Cp_2}{\nu}\right)^2 
        + 4C_A\left(\frac{\Cp_2+{\Cq_2}}{\nu}\right)^2\right]|\pi|^2.
        }\notag
    \end{align}
{Now observe that upon repeating the above analysis for $\pi=p_2-p_1$ and $\kap=q_2-q_1$, we deduce by symmetry that}
        \begin{align}
        \frac{d}{dt}|\kap|^2%\textcolor{red}{\leq 2C_L^2N^2\nu\left[\left(\frac{\Cp_2+\Cq_2}{\nu}\right)^2+\left(\frac{\Cp_1+\Cp_2}{\nu}\right)^2\right]|\kap|^2}\notag
        %\\
        %&\textcolor{red}{\quad+2N^2\nu\left[C_L^2\left(\frac{\Cp_1+\Cp_2}{\nu}\right)^2+2C_A\left(\frac{\Cp_2}{\nu}\right)^2\right]|\pi|^2.}\notag
        %\\
        &{\leq 
        2C_L^2N^2 \nu \left[\lp \frac{{\Cp_* + \Cq_*}}{\nu}\rp^2 +
        %\left(\frac{\Cp_1+\Cp_2}{\nu}\right)^2
        {\left(\frac{|\pi|}{\nu}\right)^2}\right]|\kap|^2}\notag
        \\
        &{\quad+ 2N^2\nu\left[C_L^2{\left(\frac{|\pi|}{\nu}\right)^2}
        %\left(\frac{\Cp_1+\Cp_2}{\nu}\right)^2 
        + 4C_A\left(\frac{{\Cp_*}+{\Cq_*}}{\nu}\right)^2\right]|\pi|^2.
        }\notag
    \end{align}
%{By} \cref{lem:sync}, we see that upon {bounding} $\Cp_j$ {by} $\Cp$, we {may deduce from \eqref{est:sync} that} $\Cq_j\leq \Cq$, where $\Cq$ is {given by} \eqref{def:Cq}. Thus, for $p_1,p_2\in B(\Cp)$, we have
%    \begin{align}
%        \frac{d}{dt}|\kap|^2
        %&\leq \textcolor{red}{8C_L^2N^2\nu\left[\left(\frac{\Cp+\Cq}{\nu}\right)^2+\left(\frac{\Cp}{\nu}\right)^2\right]|\kap|^2+8N^2\nu(C_L^2+C_A)\left(\frac{\Cp}{\nu}\right)^2|\pi|^2.}\notag
        %\\
        %&\leq 
        %2C_L^2N^2 \nu \left[\lp \frac{\Cp + \Cq}{\nu}\rp^2 +
        %\left(\frac{2\Cp}{\nu}\right)^2\right]|\kap|^2
        %+ 2N^2\nu\left[C_L^2\left(\frac{2\Cp}{\nu}\right)^2 + 4C_A\left(\frac{\Cp+{\Cq}}{\nu}\right)^2\right]|\pi|^2.\notag
%    \end{align}
 {Since $\kap(0)=0$}, an application of Gr\"onwall's inequality yields
     \begin{align}
        |\kap(t)|^2&\leq{2N^2}\left[{{C_L^2}\left(\frac{2\Cp}{\nu}\right)^2}+{4}C_A\right]{\nu T}\notag
        \\
        &\quad \times\exp\left\{{2}C_L^2N^2\left[\left(\frac{{\Cp_*}+{\Cq_*}}{\nu}\right)^2+{\left(\frac{2\Cp}{\nu}\right)^2}\right]\nu T\right\}\sup\limits_{t\in[0,T]}|\pi(t)|^2,\notag
    \end{align}
{which implies \eqref{est:S:lipschitz}. On the other hand, if $\kap(0)\neq0$, then we may alternatively obtain}
    \begin{align}
        %|\kap(t)|^2&\leq{\left(1+\frac{{4}C_A}{C_L^2}\right)\exp\left\{{2}C_L^2N^2\left[\left(\frac{\Cp+\Cq}{\nu}\right)^2+\left(\frac{{2}\Cp}{\nu}\right)^2\right]\nu T\right\}\sup\limits_{t\in[0,T]}|\pi(t)|^2,}\notag
        &|\kap(t)|^2\leq {\exp\left\{2C_L^2N^2\left[\left(\frac{\Cp_*+{\Cq_*}}{\nu}\right)^2+{{}\left(\frac{\sup\limits_{t\in[0,T]}|\pi(t)|}{\nu}\right)^2}\right]\nu T\right\}|\kap(0)|^2}\notag
        \\        
        &+\left({{}1}
        %{\left(\frac{\sup\limits_{t\in[0,T]}|\pi(t)|}{\nu}\right)^2}
        +\frac{{4}C_A}{C_L^2}\right) \times\exp\left\{{2}C_L^2N^2\left[\left(\frac{{\Cp_*}+{\Cq_*}}{\nu}\right)^2+{{}\left(\frac{\sup\limits_{t\in[0,T]}|\pi(t)|}{\nu}\right)^2}\right]\nu T\right\}\sup\limits_{t\in[0,T]}|\pi(t)|^2.\notag
        %\\
        %&\textcolor{blue}
        %{\leq
        %\exp\left\{8C_L^2N^2\left[\left(\frac{\Cp+\Cq}{\nu}\right)^2+\left(\frac{2\Cp}{\nu}\right)^2\right]\nu T\right\}
        %\frac{\left[\left(\frac{\Cp}{\nu}\right)^2 + \frac{C_A}{C_L^2}\left(\frac{\Cp}{\nu}\right)^2\right]}
        %{\left[\left(\frac{\Cp+\Cq}{\nu}\right)^2+\left(\frac{2\Cp}{\nu}\right)^2\right]} 
        %\sup\limits_{t\in[0,T]}|\pi(t)|^2}.\notag
    \end{align}
This implies \eqref{est:S:lipschitz:positive}, as desired.
\end{proof}

\begin{Rmk}
{{}In assessing the stability of the nonlinear terms, our analysis rewrites $B(p_1+q_1,p_1+q_1)-B(p_2+q_2,p_2+q_2)$ as \eqref{eq:B:rewrite}. This choice of representation is intentional. Indeed, the insight is that the high-mode evolution is effectively a linearized Navier-Stokes system that is forced by the low-modes. Thus, as long as the perturbations along which the nonlinearity is linearized, i.e., $DB(v_2)\kap$, $DB(\pi)\kap$ are bounded, then all other terms which remain in the subsequent energy balance, i.e., $B(\pi,\pi), DB(v_2)\pi$ are viewed as forcing terms.
}
\end{Rmk}
\begin{comment}
We are now ready to prove the main theorem, \cref{thm:main}, from \cref{sect:intro}.

%\begin{Thm}\label{thm:converge}
%Suppose that $u_0\in B_H(\rho_0)\cap B_V(\rho_1)$ and %$q_0\in V$. Then for any $T,N>0$ it holds that
%    \begin{align}\label{eq:converge}
%        \lim\limits_{\mu\goesto\infty}\sup\limits_{t\in[0,T]}|\tv(t;\tv_0)-v(t;v_0)|=0,
%    \end{align}
%whenever $P_N\tv_0=P_Nv_0$.
%\end{Thm}

\begin{proof}[Proof of \cref{thm:main}]
Let $p=P_Nu$, where $u(\cdotp;u_0)$ is the unique global strong solution of \eqref{eq:nse:ff} corresponding to $u_0$. By \eqref{est:abs:ball}, $p(\cdotp)\subset P_NB_H(\rho_0)$.  In particular, $|p| \leq \rho_0$. Let $\tq=\sS(\tp;q_0)$ and $q=\sS(p;q_0)$, so that ${\tv}=\tp+\tq$ represents the unique solution of \eqref{eq:nse:nudge:high:low} corresponding to $\tv_0=p_0+q_0$ and $v=p+q$ represents the unique solution of \eqref{eq:nse:sync:high:low} corresponding to $v_0=p_0+q_0$. By \cref{lem:y}, it follows that
    \begin{align}\label{eq:converge:low}
        \lim\limits_{\mu\goesto\infty}\sup\limits_{t\in[0,T]}|\tp(t;p_0)-p(t;p_0)|=0.
    \end{align}
Hence, $p(\cdotp), \tp(\cdotp)\subset P_NB_H(\rho_0)$, for $\mu$ sufficiently large. By \cref{thm:S:lipschitz}, it follows that
    \begin{align}\label{eq:converge:high}
        \lim\limits_{\mu\goesto\infty}\sup\limits_{t\in[0,T]}|\tq(t;\tp)-q(t;p)|=0.
    \end{align}
By orthogonality, \eqref{eq:converge:low} and \eqref{eq:converge:high} imply \eqref{eq:main:claim2}.
\end{proof}

\end{comment}

{We are now in a position to} prove \cref{thm:nudge:limit}, which is a refinement of \eqref{eq:main:claim2} {from \cref{thm:main} wherein we quantify} the relationship between the time window size $T$ and nudging parameter $\mu$. 

\begin{proof}[Proof of \cref{thm:nudge:limit}]

To begin, suppose that $\mu\geq \nu$. Recall that we assume $v_0=p_0+q_0$, where $p_0=P_Nu_0$. Thus $\sup\limits_{t\in[0,T]}|p(t)|\leq\rho_1$. To invoke \cref{thm:S:lipschitz}, let $\Cp_0=2(\rho_1+|\tp_0|+\nu\tC)$, where $\tC$ is defined by \eqref{def:tC}. Note that $\Cp_0$ is independent of $\mu$. Then by \cref{lem:y} we have $\sup\limits_{t\in [0,T]}|p(t;p_0)|$, $\sup\limits_{t\in[0,T]}|\tp(t;\tp_0)|<\Cp_0$. Moreover, the appearance of {{}the minimum in} $\Cq_*$ in \eqref{est:S:lipschitz} allows us to invoke \cref{lem:sync} for $q=\sS(p;q_0)$, rather than $\sS(\tp;\tq_0)$, so that we may choose $\Cq_*=\Cq$, where $\Cq$ is defined by \eqref{def:Cq}.

First, we treat the case $\tv_0=v_0$. {{}In particular $\tp_0=p_0$}. Then upon combining \cref{lem:y} and \cref{thm:S:lipschitz} \eqref{est:S:lipschitz}, we deduce
    \begin{align}\label{est:converge:refined}
        \sup\limits_{t\in[0,T]}|\tq(t;\tp)-q(t;p)|^2\leq {C_{\sS}^2(\nu T)\tC(\tq_0,{{}Qu_0},T)^2}\frac{\nu^3}{\mu},
    \end{align}
where $C_{\sS}, \tC$ are defined by \eqref{def:Cs}, \eqref{def:tC}, respectively. {Now fix $0<\eps<1$ and choose $T:=T(\mu)$} such that the right-hand side of \eqref{est:converge:refined} is $O((\nu/\mu)^{\eps})$. More precisely, $C_\sS^2(\nu T)\tC^2\sim \exp\exp(C\nu T)$, for some sufficiently large constant $C>0$, depending on $\rho_0, \rho_1, q_0, N, G$. Thus, we let $T\sim (C\nu)^{-1}\ln\ln (\mu/\nu)^{1-\eps}$. For this choice of $T$, we have { from \eqref{est:converge:refined} that}
    \begin{align}\notag
        \lim\limits_{\mu\goesto\infty}\sup\limits_{t\in[0,T(\mu)]}|\tq(t;\tp)-q(t;p)|=0,
    \end{align}
which implies \eqref{est:nudge:infinity} {for $t_0=0$}, as claimed.

{Now suppose that $\tq_0=q_0$, but $\tp_0\neq p_0$. Given any $t_0>0$, we combine \cref{lem:y} and \cref{thm:S:lipschitz} \eqref{est:S:lipschitz} to deduce
    \begin{align}\label{est:converge:refined:transient}
    |\tq(t_0;\tp)-q(t_0;p)|^2\leq C_{\sS}^2(\nu t_0)\left(|\tp_0-p_0|^2+\tC(\tv_0,u_0,t_0)^2\frac{\nu^3}{\mu}\right),
    \end{align}
Then for $t\in[t_0,T]$, we invoke \cref{thm:S:lipschitz} \eqref{est:S:lipschitz:positive} to obtain
    \begin{align}
        |\tq(t;\tp)-q(t;p)|^2\leq \tilde{C}_{\sS,1}^2|\tq(t_0;\tp)-q(t_0;p)|^2+\tilde{C}_{\sS,2}^2\sup\limits_{t\in[t_0,T]}|p_1(t)-p_2(t)|^2,
    \end{align}
and invoke \cref{lem:y} over $[t_0,T]$, so that
    \begin{align}
        \tilde{C}_{\sS,1}^2&\leq \exp\left\{2C_L^2N^2\left[\left(\frac{\Cp+{\Cq_*}}{\nu}\right)^2+{e^{-2\mu t_0}\left(\frac{|\tp_0-p_0|}{\nu}\right)^2+ \frac{\nu}{\mu}\tC(\tv_0,u_0,T)^2}\right]\nu T\right\}\notag
        \\
        \tilde{C}_{\sS,2}^2&\leq \left(
        %{e^{-2\mu t_0}\left(\frac{|\tp_0-p_0|}{\nu}\right)^2+ \frac{\nu}{\mu}\tC(\tv_0,u_0,T)^2}
        {{}1}+\frac{{4}C_A}{C_L^2}\right)\tilde{C}_{\sS,1}^2.\notag
    \end{align}
Therefore
    \begin{align}
        |\tq(t;\tp)-q(t;p)|^2&\leq \tilde{C}_{\sS,1}^2(\nu t_0)\left(|\tp_0-p_0|^2+\tC(\tv_0,u_0,t_0)^2\frac{\nu^3}{\mu}\right)\notag
        \\
        &\quad +\tilde{C}_{\sS,2}^2\left(e^{-2\mu t_0}\left(\frac{|\tp_0-p_0|}{\nu}\right)^2+ \frac{\nu}{\mu}\tC(\tv_0,u_0,T)^2\right)\notag.
    \end{align}
It follows that
    \begin{align}
        \limsup\limits_{\mu\goesto\infty}\sup\limits_{t\in[t_0,T]}|\tq(t;\tp)-q(t;p)|^2\leq \exp\left\{2C_L^2N^2\left(\frac{\Cp+\Cq_*}{\nu}\right)^2\nu T\right\}(\nu t_0)|\tp_0-p_0|^2.\notag
    \end{align}
Sending $t_0\goesto0^+$ then completes the proof.
}
\end{proof}

{We conclude the section with a proof of \cref{thm:nudge:converge}. We recall that the goal here is to show that for $N$ sufficiently large and $\mu$ sufficiently large relative to $N$, but finite, one achieves asymptotic convergence of $\tv$ to $u$.}

\begin{proof}[Proof of \cref{thm:nudge:converge}]
{
Let $w=\tv-u$. Then \eqref{eq:w} holds.
%    \begin{align}
%        \frac{dw}{dt}+\nu Aw+B(w,w)+DB(u)w=-\mu P_Nw,\quad w(0)=\tv_0-u_0.\notag
%    \end{align}
Upon taking the $H$-inner product of \eqref{eq:w} with $w$ and invoking \eqref{eq:B:skew}, we obtain
    \begin{align}\label{eq:z:balance}
        \frac{1}2\frac{d}{dt}|w|^2+\nu \|w\|^2+\mu|P_Nw|^2=\lp B(w,u),w\rp.
    \end{align}
Decomposing into low-modes and high-modes, we obtain
    \begin{align}
        \lp B({{}w}, u), {{}w}\rp&=\lp B({{}P_Nw}, u), {{}P_Nw}\rp+\lp B({{}P_Nw}, u), {{}Q_Nw}\rp\notag
        \\
        &\quad+\lp B({{}Q_Nw}, u), {{}P_Nw}\rp+\lp B({{}Q_Nw}, u), {{}Q_Nw}\rp\notag.
    \end{align}
We estimate each term using \eqref{est:interpolation}, \eqref{est:interpolation:CS}, \eqref{est:Bernstein}, and Young's inequality:
    \begin{align}
        |\lp B({{}w}, u),{{}w}\rp|&\leq C_L\|{{}P_Nw}\||{{}P_Nw}|\|u\|+C_L\|{{}Q_Nw}\||{{}Q_Nw}|\|u\|\notag
        \\
        &\quad+2C_L\|{{}P_Nw}\|^{1/2}|{{}P_Nw}|^{1/2}\|u\|\|{{}Q_Nw}\|^{1/2}|{{}Q_Nw}|^{1/2}\notag
        \\
        &\leq 3C_L\rho_1\|{{}w}\||{{}P_Nw}|+\frac{C_L\rho_1}{N}\|{{}w}\|^2\leq \nu\left(\frac{1}4+\frac{C_L\rho_1}{N\nu}\right)\|{{}w}\|^2+\frac{9C_L^2\rho_1^2}{\nu}|{{}P_Nw}|^2\notag.
    \end{align}
Upon returning to \eqref{eq:z:balance} and invoking the assumptions on $N$ and $\mu$, it follows that
    \begin{align}\notag
        \frac{d}{dt}|{{}w}|^2+\nu\|{{}w}\|^2\leq 0.
    \end{align}
An application of Poincar\'e's inequality and Gr\"onwall's inequality yields \eqref{est:nudge:converge}.
}
\end{proof}

\begin{Rmk}\label{rmk:main}
{
For $\tv$ satisfying \eqref{eq:nudge}, one only has the a priori bound in general:
    \begin{align}\label{est:v:apriori}
        |\tv(t)|^2\leq e^{-\nu t}|\tv(0)|^2+\nu^2\Gr^2+\mu\Cp^2,
    \end{align}
which clearly blows up as $\mu\goesto\infty$. Upon choosing $N$ sufficiently large, depending on $\mu$, one may obtain a priori bounds independent of $\mu$. However, such an assumption obviously precludes taking $\mu\goesto\infty$ independent of $N$. Notably, in this regime, one may show that $\tv\goesto u$ as $t\goesto\infty$. Indeed, this is the main result of \cite{AzouaniOlsonTiti2014} (see \cref{thm:AOT}).

If one considers the low-mode evolution of $\tv$, then it is, in principle, possible that the high-modes may sufficiently excite the low-modes through the nonlinearity, \textit{in spite} of the presence of the stabilizing mechanism provided by nudging in the low-mode evolution, and in such a way that saturates the bound \eqref{est:v:apriori}. Thus, one of the accomplishments in proving \cref{thm:main} is firstly, the recognition that this pessimistic scenario can be circumvented by instead directly studying the low-mode error. In particular, we show through \cref{lem:w:bad} and \cref{lem:y} that the nudging feedback control does indeed provide sufficient stabilization of the low-mode errors. Secondly, we show that this low-mode stabilization can effectively be conferred to the high-mode evolution of the errors via the local Lipschitz property of the mapping $\sS$, provided that one is confined to a finite, but arbitrarily sized, time-horizon, and in a manner that is independent of $N$.
%the low-modes of the nudged Navier-Stokes solution, $\tv$, converges to the low-modes of the Navier-Stokes, $\tu$, uniformly-in-time whenever the systems are initialized identically on the low-mode subspace (\cref{lem:y}). 
{{}It is crucial that in the domain of the map $\sS$, that is in the low-modes, the Lipschitz property is posed on a ball} %of finite-dimensionally determined trajectories 
{{}of sufficiently large radius} to contain all possible asymptotic behavior of the observations (see \cref{thm:S:lipschitz}). This provides for the freedom to evaluate $\sS$ along both $p$ and $\tp$ or, in other words, to not distinguish between whether the argument arises as low-modes of the NSE or of its corresponding nudged equation. 

%Notably, this local Lipschitz property of $\sS$ is genuinely nonlinear;
%the velocity high-mode evolution governed by \eqref{eq:sync} is Lipschitz in a ball of finite-dimensionally determined trajectories that is sufficiently large enough to contain all possible asymptotic behavior of the observations (\cref{thm:S:lipschitz}). 
%were the low-mode and high-mode evolutions not coupled through nonlinearity, such a property would not be possible.  Indeed, 

%Therefore, it is the presence of nonlinearity in \eqref{eq:nse}, facilitated through the local Lipschitz property of $\sS$, that ultimately allows one to overcome the potentially destabilizing effects of driving a nonlinear system towards the observations in an increasingly singular fashion through nudging.
}

%\textcolor{green}{Liz: make a remark on the fourier support of $f$ and how this affects the dimensionality of the attractor
%}
%{\color{blue}Vincent: We don't actually make use of finite-dimensionality of the dynamics in obtaining our finite-time convergence result, so it may best to only mention this to referee if needed and leave it out of the manuscript.}
\end{Rmk}

\begin{Rmk}\label{rmk:heat}
{An important point of comparison for our results is the linear heat equation. In this case, the corresponding nudged variable, $\tv$, solves
    \begin{align}\label{eq:heat:nudge}
        \bdy_t\tv+\nu A\tv=f-\mu \tp+\mu p,
    \end{align}
Then $\tv$ enjoys a priori bounds that are independent of $\mu$ owing to linearity and orthogonality:
    \begin{align}
        |\tp(t)|^2\leq e^{-\mu t}|\tp(0)|^2+\frac{\nu^3}{\mu}\Gr^2+\Cp^2,\quad
        |\tq(t)|^2\leq e^{-\nu t}|\tq(0)|^2+\nu^2\Gr^2\notag.
    \end{align}
Due to linearity, these a priori bounds directly imply $\tp\goesto p$ as $\mu\goesto\infty$.
%\textcolor{green}{; indeed,
%\begin{align}
%    \partial_t (\tilde{p}-p) + \nu A(\tilde{p}-p) &= f - \mu (\tilde{p}-p)\\
%    \partial_t (\tilde{q}-q) + \nu A(\tilde{q}-q) &= f
%\end{align}
%which implies 
%\begin{align}
%    |\tilde{p}(t)-p(t)|^2 &\leq e^{-(\mu+\nu)t} |\tilde{p}(0)-p(0)|^2
%    \\
%    |\tilde{q}(t)-q(t)|^2 &\leq e^{-\nu t} |\tilde{q}(0)-q(0)|^2.
%\end{align}
%}
%In this context, the high-modes of the {direct-replacement algorithm} are identical to the high-modes of the {nudging algorithm} whenever they are initially equal. 
Due to orthogonality, the high-modes evolve independently of the low-modes; this automatically implies the impossibility of the high-mode errors to converge to zero within any finite-time window in the infinite-nudging limit unless the high-modes are initially identical.  Nevertheless, by dissipativity of the system, the high-mode errors asymptotically converge to zero. 
%\textcolor{green}{as $t \to \infty$}. 
%\textcolor{green}{Note that these equations illuminate how $\mu\to \infty$ in the $p-\tilde{p}$ equation produces the same convergence result as $t\to\infty$ in the $p-\tilde{p}$ equation.}

From this point of view, \cref{thm:main} constitutes a nontrivial extension to the nonlinear setting represented by \eqref{eq:nse} in a paradigmatic way. This extension is primarily facilitated by establishing the local Lipschitz property of the mapping $\sS$, which is what allows one to transfer low-mode convergence (a linear phenomenon) to high-mode convergence on arbitrary finite-time horizons, an ultimately nonlinear phenomenon. On the other hand, on a given bounded set of low-mode trajectories, the system \eqref{eq:sync} is not guaranteed to be dissipative unless $N$ is taken sufficiently large; this observation, in conjunction with the local Lipschitz property of $\sS$, is what allows us to additionally establish global-in-time convergence.
}

%\textcolor{green}{
%\cref{thm:main} constitutes a nontrivial extension of this linear convergence to the nonlinear setting of \eqref{eq:nse} in the context of \eqref{eq:nudge:AOT}.  This extension is primarily facilitated by establishing the local Lipschitz property of the mapping $\sS$, which is a genuinely nonlinear property in this setting, and ultimately what allows us to mimic the convergence behavior of the high modes in the dissipative linear setting (given enough modes, $N$).  Proving the direct-replacement algorithm \eqref{eq:sync} is dissipative is known only to hold for $N$ sufficiently large, due to lack of cancellation of the bilinear terms in the energy estimates.  These two observations are what allow us to to establish global-in-time convergence of the nudging algorithm to the direct-replacement algorithm.
%}
\end{Rmk}

\begin{Rmk}
The recent works \cite{LiHawkinsRebholzVargun2023, DiegelLiRebholz2024} also studied the effect of large $\mu$ in the context of finite element discretizations of the 2D NSE. It is shown in \cite{DiegelLiRebholz2024} that the error analysis of the discretization scheme can be made to be independent of the nudging parameter. However, their analysis does not establish a convergence result in the passage to the infinite-$\mu$ limit to a corresponding discretization of the {direct-replacement algorithm}. Nevertheless, comprehensive numerical tests are carried in both works demonstrating convergence of the numerical approximation to the assumed observed values. In comparison, the results of our numerical results carried out with pseudo-spectral methods are consistent the results in \cite{LiHawkinsRebholzVargun2023, DiegelLiRebholz2024}.
\end{Rmk}

\section{Computational Results}\label{sect:numerical}

\subsection{Numerical Methods}

Simulations of the 2D Navier-Stokes equations are performed in MATLAB (R2023b) using a fully dealiased pseudo-spectral code defined on the periodic box $\mathbb{T}^2 = [-\pi,\pi]^2$. That is, the spatial derivatives were calculated by multiplication in Fourier space. The equations were simulated at the stream function level, i.e. the 2D Navier-Stokes equations were written in the following form:
\begin{align}\label{scheme:NSE}
    \begin{split}
\psi_t + \De^{-1}(\nabla^{\perp}\psi\cdot\nabla)\De\psi &= \nu\De\psi + \De^{-1}\nabla^{\perp}\cdot f,
%\omega &= \nabla^\perp\cdotp u= -\De \psi.
    \end{split}
\end{align}
where $\nabla^\perp=(-\bdy_y,\bdy_x)$ and $\De^{-1}$ denotes the inverse Laplacian, which is taken with respect to the periodic boundary conditions and the mean-free condition. The initial condition and parameters were chosen as in \cite{Franz_Larios_Victor_2021} such that our simulations coincide with a turbulent regime. Specifically, the viscosity, $\nu$ was chosen to be $\nu = 0.0001$, and the body force chosen as in \cite{OlsonTiti2008} to be low-mode forcing concentrated over a band of frequencies with $10\leq |\vec{k}|^2 \leq 12$. The forcing term is normalized such that the Grashof number $G = \frac{\|f\|_{L^2}}{\nu^2} = 500,000$.  The spatial resolution utilized for our simulations is $N = 2^{10}$, which yields $341.\overline{3}$ active Fourier modes. The initial data was generated by running the 2D NSE solver forward with zero initial condition out to time $t = 10,000$. The spectrum of the initial data can be seen in \cref{fig:spectrum}, we note that the initial profile is well-resolved, as the energy spectrum decays to machine precision (approximately $2.2216 e$--$16$) before the $2/3$ dealiasing line; all simulations presented within this work remain well-resolved for the duration of each simulation.

The time-stepping scheme we utilized was a semi-implicit scheme, where we handle the linear diffusion term implicitly via an integrating factor in Fourier space. For an overview of integrating factor schemes see e.g.\cite{Kassam_Trefethen_2005, Trefethen_2000_spML} and the references contained within. The equations are then evolved using an implicit Euler scheme, with the nonlinear term being treated explicitly and the AOT feedback-control term implicitly.

\begin{figure}
    \centering        
    % \missingfigure[\width=.85\linewidth]
    \includegraphics[width=0.85\linewidth]{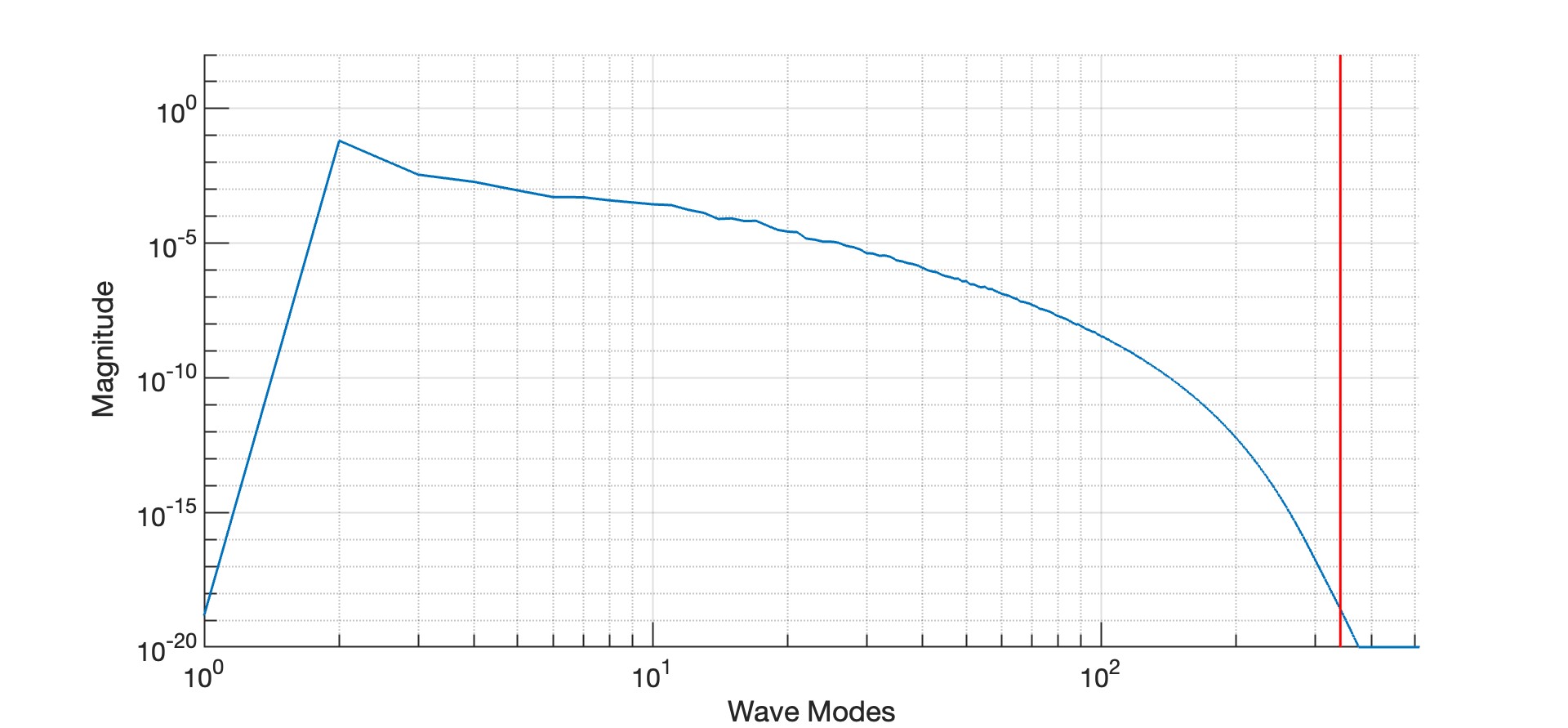}
    \caption{Energy spectrum of the initial data with $\nu = 0.0001$, $G = 500,000$, and $\Delta t = 0.001$. The vertical red line is the 2/3 dealiasing cutoff as $\frac{2}{3}\frac{N}{2} = 341.\overline{3}.$}
    \label{fig:spectrum}
\end{figure}

We emphasize that implicit Euler for the time-stepping was chosen for its simplicity and its ability to accommodate large values of $\mu$, which is crucial for this study. The implicit treatment of the AOT feedback-control term allowed us to relax the CFL condition that is typically present in explicit implementations as $\mu \lesssim \frac{2}{dt}$, where $dt$ is size of the time step, thus facilitating the exploration of the $\mu \goesto \infty$ regime.

In order to test each method of data assimilation described in the previous sections, we performed a series of ``identical twin'' experiments. These experiments are commonly used to test methods of data assimilation and involve running two separate simulations, one for the ``truth'' and one which uses data assimilation to attempt to recover said truth. These simulations are run in the same time loop, with the reference solution being utilized to generate observational data for the data assimilation process.

Additionally, in our study, we examine two different paradigms of observational data, deterministic and noisy observations. The deterministic observations are direct observations of the true state, whereas the noisy observations are polluted with Gaussian white noise. For details on how the noisy observations are generated, see \cref{sect:noisy_obs}.

\subsection{Convergence of Nudging to Synchronization -- Deterministic Observations}

In this section we describe the results of various numerical tests comparing rates of convergence across a range of $\mu$-values. In all of the trials discussed below, we initialize all schemes with identical observational data that is given by a low-mode Fourier projection of the truth solution. We note that while all methods should work with arbitrary initial conditions, we initialize all tests with the first set of observational data in order to establish initialization as a control variable for our experiments. 
%in order to even the playing field between each method.

%We examine the convergence rates for various values of $\mu$ with deterministic observations. 
In \cref{fig:det:obs:resolved}, we see that when using observational data from the lowest 100 Fourier modes, we obtain convergence to machine precision for all $\mu \geq 1$. In contrast, (see \cref{fig:det:obs:unresolved}) when too few Fourier modes are observed, we \textit{do not} obtain convergence to the reference solution. Regardless, as $\mu$ increases, we see an improvement in the observed error, thus confirming the theoretical analysis. Indeed, as we send the value of $\mu$ towards infinity, the observed error converges to the observed error of the {direct-replacement algorithm} to machine precision

When enough Fourier modes are observed we see the expected behavior for both the synchronization and nudging schemes. That is, we see in \cref{fig:det:obs:resolved} that all methods exhibit exponential convergence in time to the reference solution in both the observed and unobserved errors. Moreover, in \cref{fig:initialized zero} we see that this same behaviour occurs when the initial data for the nudging scheme is chosen to be something other than the reference solution, in this case we use zero initial data. We see in \cref{fig:initialized zero} that when we initialize the nudged equations with something different than the observations at the initial time, {the effect of nudging on the observed error is felt during the initial period, where the value of $\mu$ is found to determine the initial rate of rapid convergence, as well as the error level at which the error transitions from the initial rate to a slower but still exponential rate of decay.  We note that this initial rapid convergence at exponential rate $\approx \mu$ is expected, due to \cref{lem:y} and \cref{thm:nudge:converge}. In particular \cref{lem:y} guarantees initially rapid low mode convergence $$\norm{P_M u(t) - P_M \tv(t)}_{L^2}^2 \leq e^{-2\mu t}\norm{P_M u_0 - P_M \tv_0}_{L^2}^2 + \frac{\nu^3}{\mu}C(\tv_0, u_0, T),$$ while \cref{thm:nudge:converge} guarantees slower convergence, but on an infinite time horizon $$\norm{P_M u(t) - P_M \tv(t)}_{L^2}^2 \leq e^{-\nu(t-T)} \norm{u(T) - \tv(T)}_{L^2}^2.$$ When $\mu$ is sufficiently large, these results combine in a straightforward way to achieve the bound $$\norm{P_M u(t) - P_M \tv(t)}_{L^2} \leq O(e^{-\mu t})+O(e^{-\frac{\nu}2(t-T)}),$$ which yields the behavior observed in the \cref{fig:initialized zero}.
}

We also investigated the complementary limit, as $\mu$ decreases to $0$, the results of which are presented in \cref{fig:zero limit}. We point out that in \cref{fig:zero limit}, the error plotted is not $\norm{u -\tv}_{L^2}$, as it is in all other plots. Instead the error here is $\norm{\tu - \tv}_{L^2}$, where $\tu$ is the 2D NSE solution that is initialized at $t_0$ with \textit{zero initial data}, {so that $\tv$ and $\tu$ are initialized identically. Note that because of this, one explicitly recovers $\tu$ when $\mu = 0$ (see \cref{thm:nudge:zero:limit})}. When $\mu \neq 0$ we can see that the error in both observed and unobserved modes increases at each fixed time as one increases the value of $\mu$, again, consistent with expectation.

\begin{figure}
    \begin{subfigure}[b]{.45\textwidth}
\centering

        % \begin{minipage}
            \includegraphics[width=\textwidth,height=5cm]{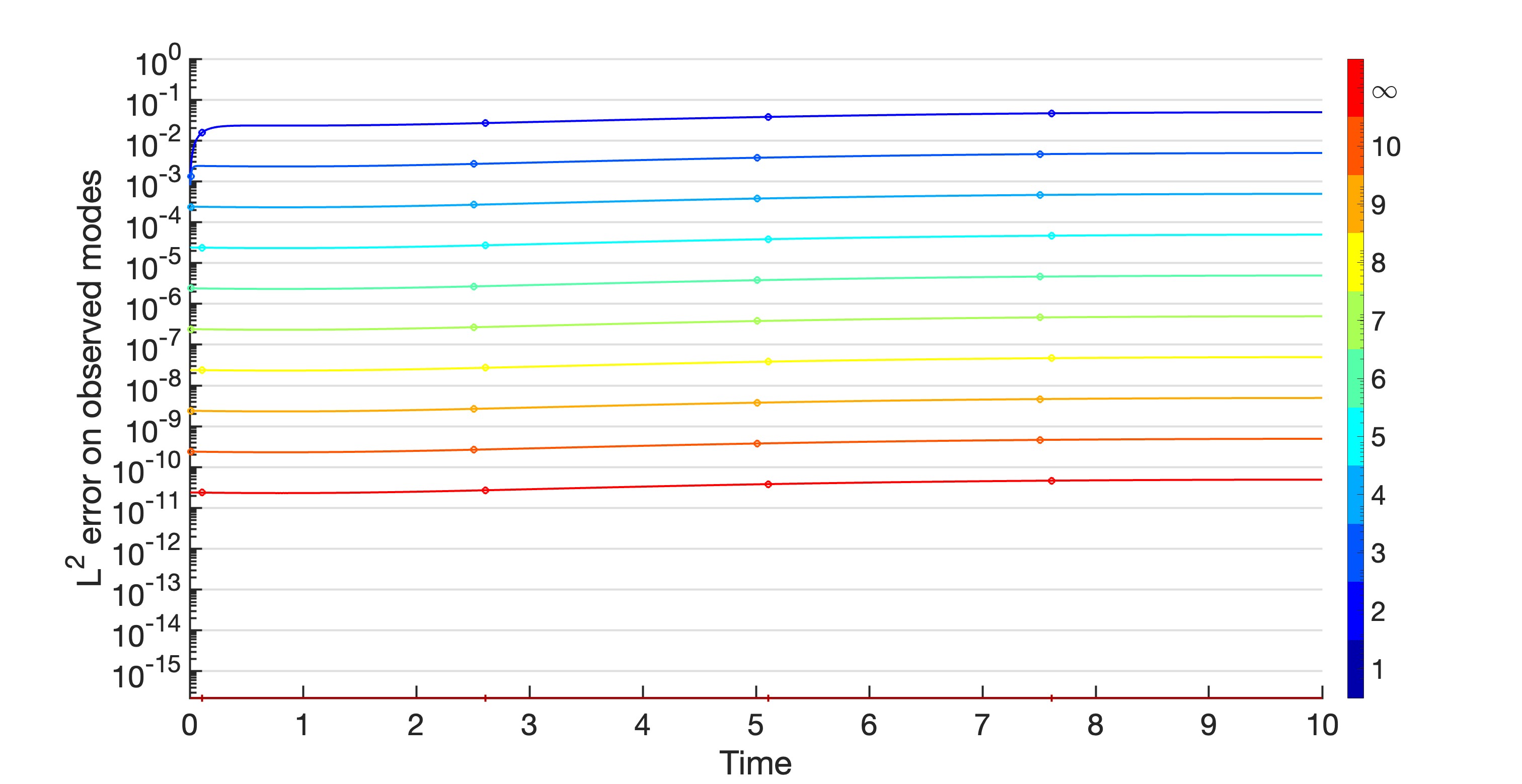}

    % \end{minipage}
    % \caption{Caption}
    \end{subfigure}
    \begin{subfigure}[b]{.45\textwidth}
\centering

        % \begin{minipage}
            \includegraphics[width=\textwidth,height=5cm]{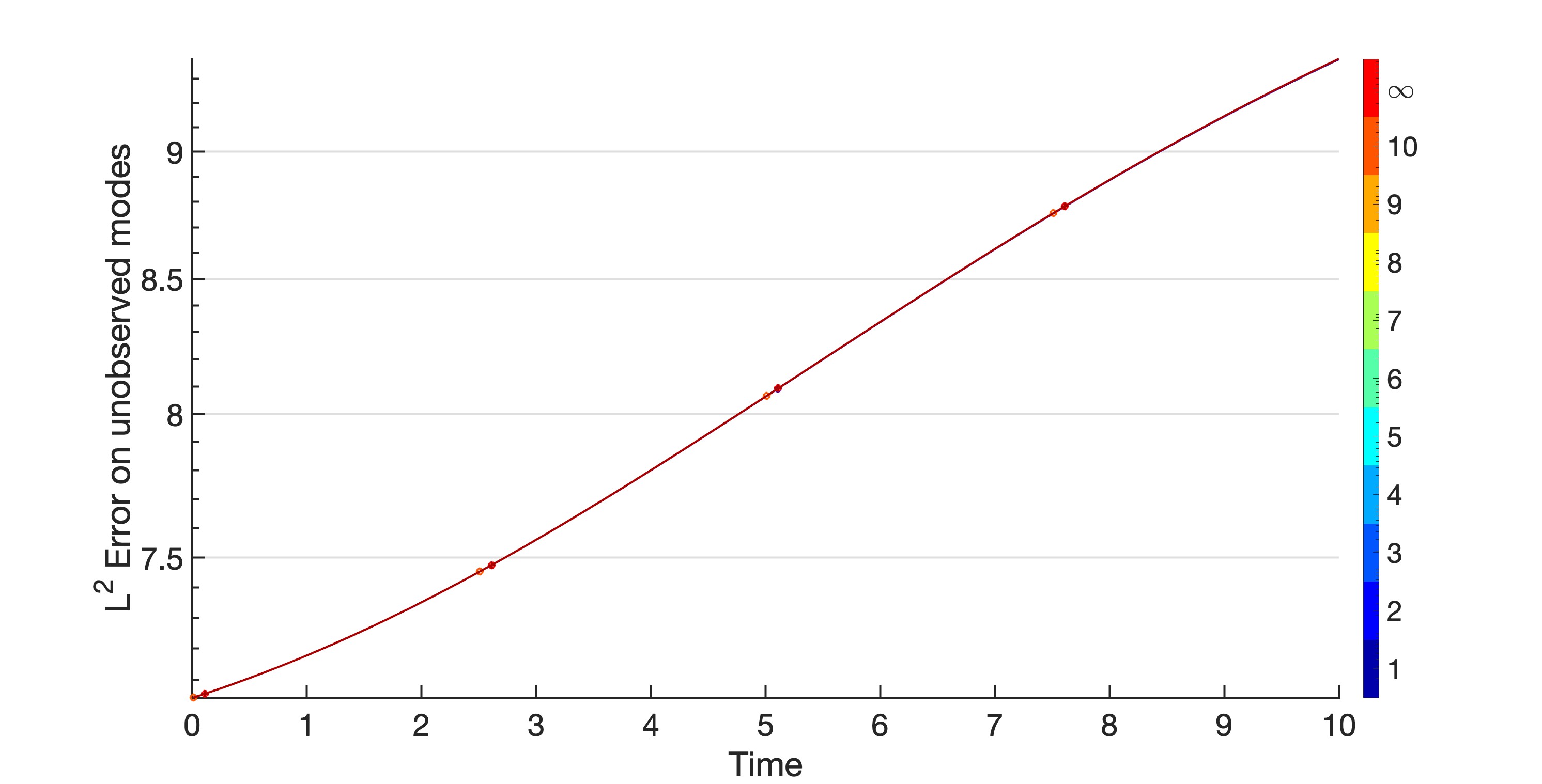}

    % \end{minipage}
    % \caption{Caption}
    \end{subfigure}
%     \begin{subfigure}[b]{.32\textwidth}
% \centering

%         % \begin{minipage}
%             \includegraphics[width=\textwidth,height=5cm]{error_total_2.jpg}

%     % \end{minipage}
%     % \caption{Caption}
%     \end{subfigure}
    \caption{\small Error over time between reference and nudging solutions for different $\mu$ values when lowest 2 modes are observed. Plotted errors probe infinite-$\mu$ limit and display low mode error (left) and high mode error (right) in $L^2$ norm. {Note that direct-replacement algorithm} ($\mu=\infty$) achieves smallest total error among all values of $\mu$.  Coloring corresponds to values $\mu = 10^k$, with $k$ indicated by the color bar; $k= \infty$ and $k= -\infty$ correspond to the direct-replacement and zero-nudging regime, respectively.}
    \label{fig:det:obs:unresolved}
\end{figure}

% \begin{figure}
%     \centering
%     \includegraphics[width=\textwidth,height=5cm]{high_error.jpg}
%     \caption{Error over time for lowest 100 modes observed.}
%     \label{fig:det:obs:resolved}
% \end{figure}

\begin{figure}
    \begin{subfigure}[b]{.45\textwidth}
\centering

        % \begin{minipage}
            \includegraphics[width=\textwidth,height=5cm]{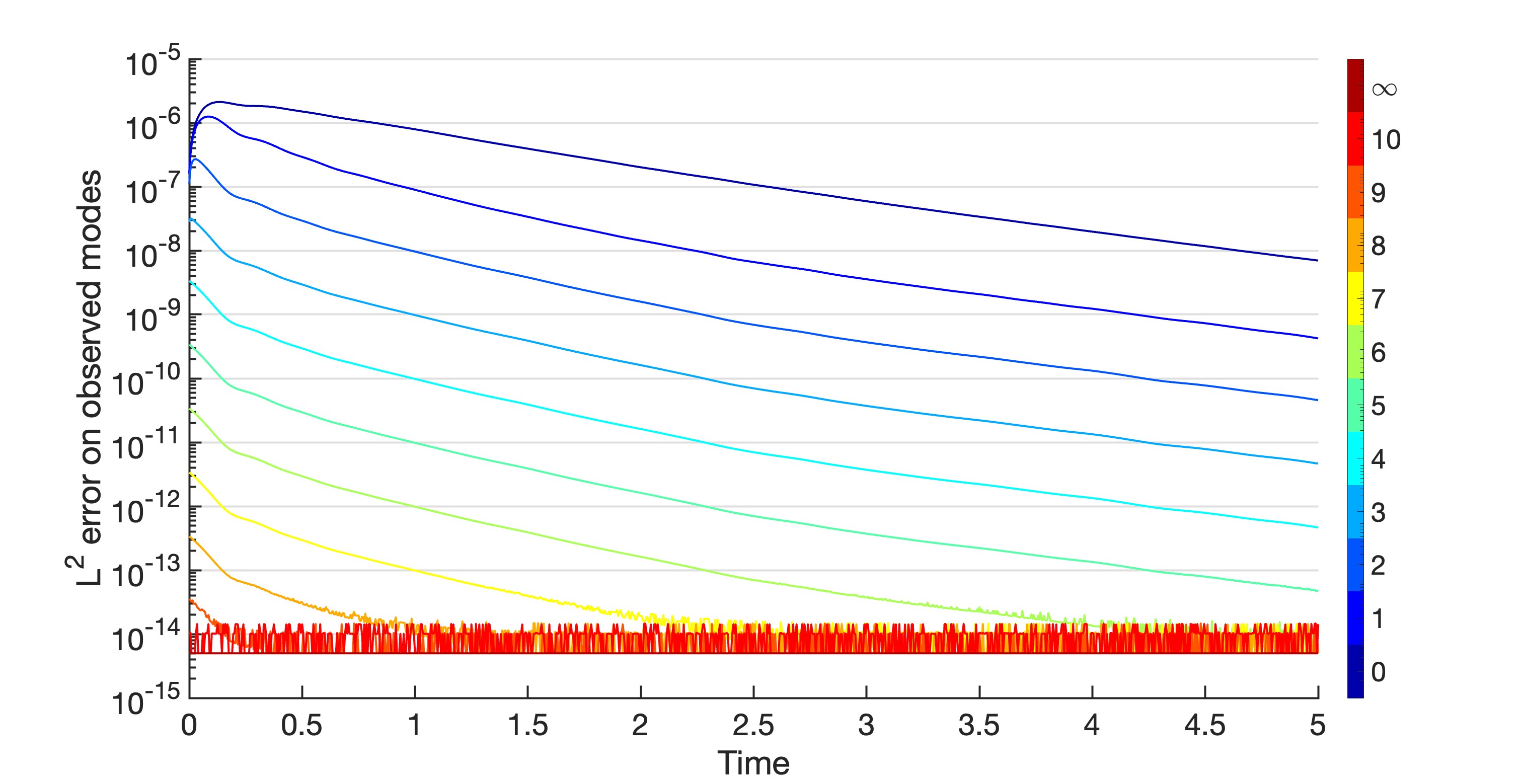}

    % \end{minipage}
    % \caption{Caption}
    \end{subfigure}
    \begin{subfigure}[b]{.45\textwidth}
\centering

        % \begin{minipage}
            \includegraphics[width=\textwidth,height=5cm]{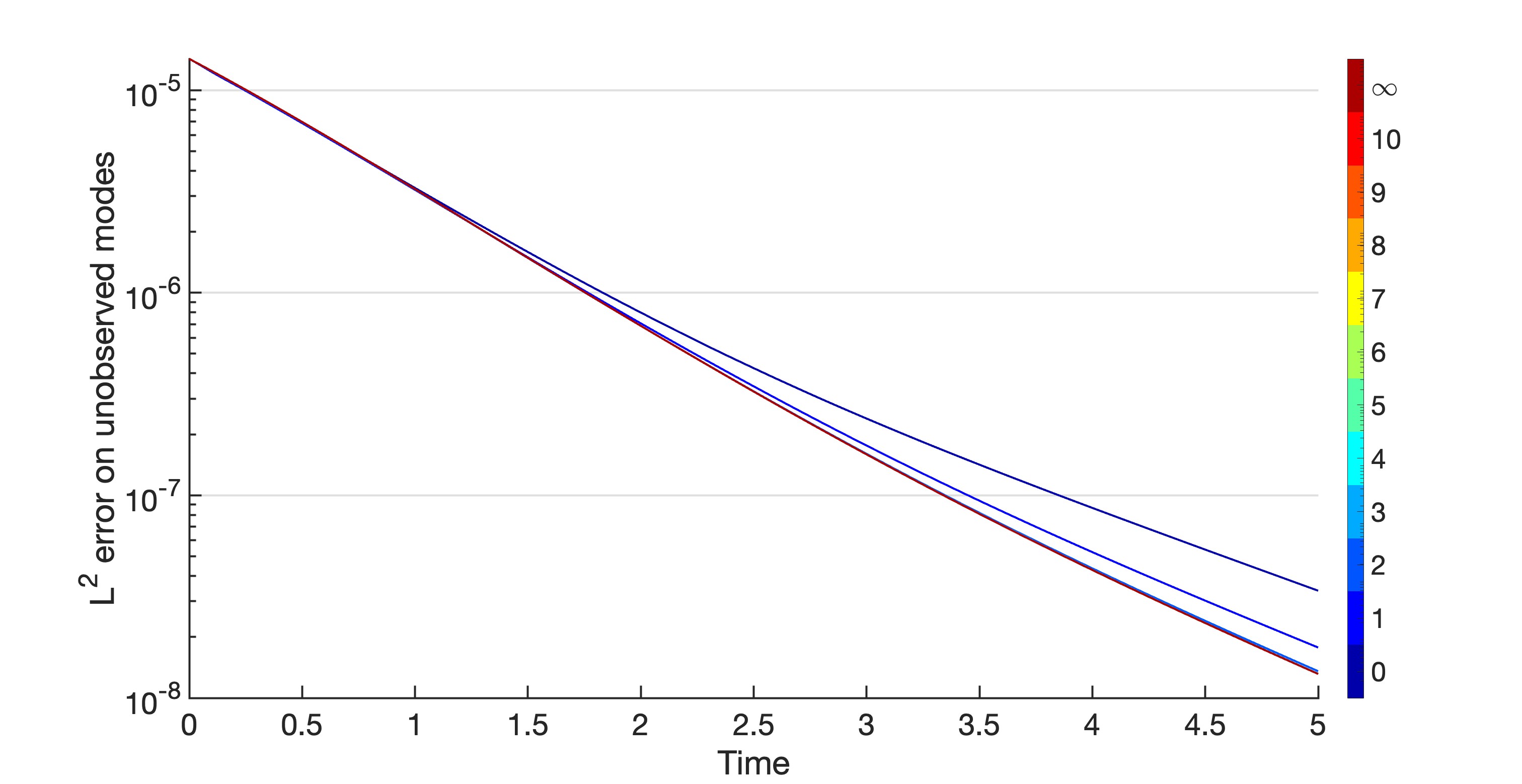}

    % \end{minipage}
    % \caption{Caption}
    \end{subfigure}
%     \begin{subfigure}[b]{.32\textwidth}
% \centering

%         % \begin{minipage}
%             \includegraphics[width=\textwidth,height=5cm]{error_total_100.jpg}

%     % \end{minipage}
%     % \caption{Caption}
%     \end{subfigure}

    \caption{\small Error over time between reference and nudging solutions for different $\mu$ values when lowest 100 modes are observed. Nudging algorithm was initialized with low modes of the observed reference solution. Plotted errors probe infinite-$\mu$ limit and display low mode error (left) and high mode error (right) in $L^2$.  Coloring corresponds to values $\mu = 10^k$, with $k$ indicated by the color bar; $k= \infty$ and $k= -\infty$ correspond to the direct-replacement and zero-nudging regime, respectively.}
    \label{fig:det:obs:resolved}
\end{figure}

\begin{figure}
    \begin{subfigure}[b]{.45\textwidth}
\centering

        % \begin{minipage}
        % \missingfigure[\width=\textwidth]
            \includegraphics[width=\textwidth,height=5cm]{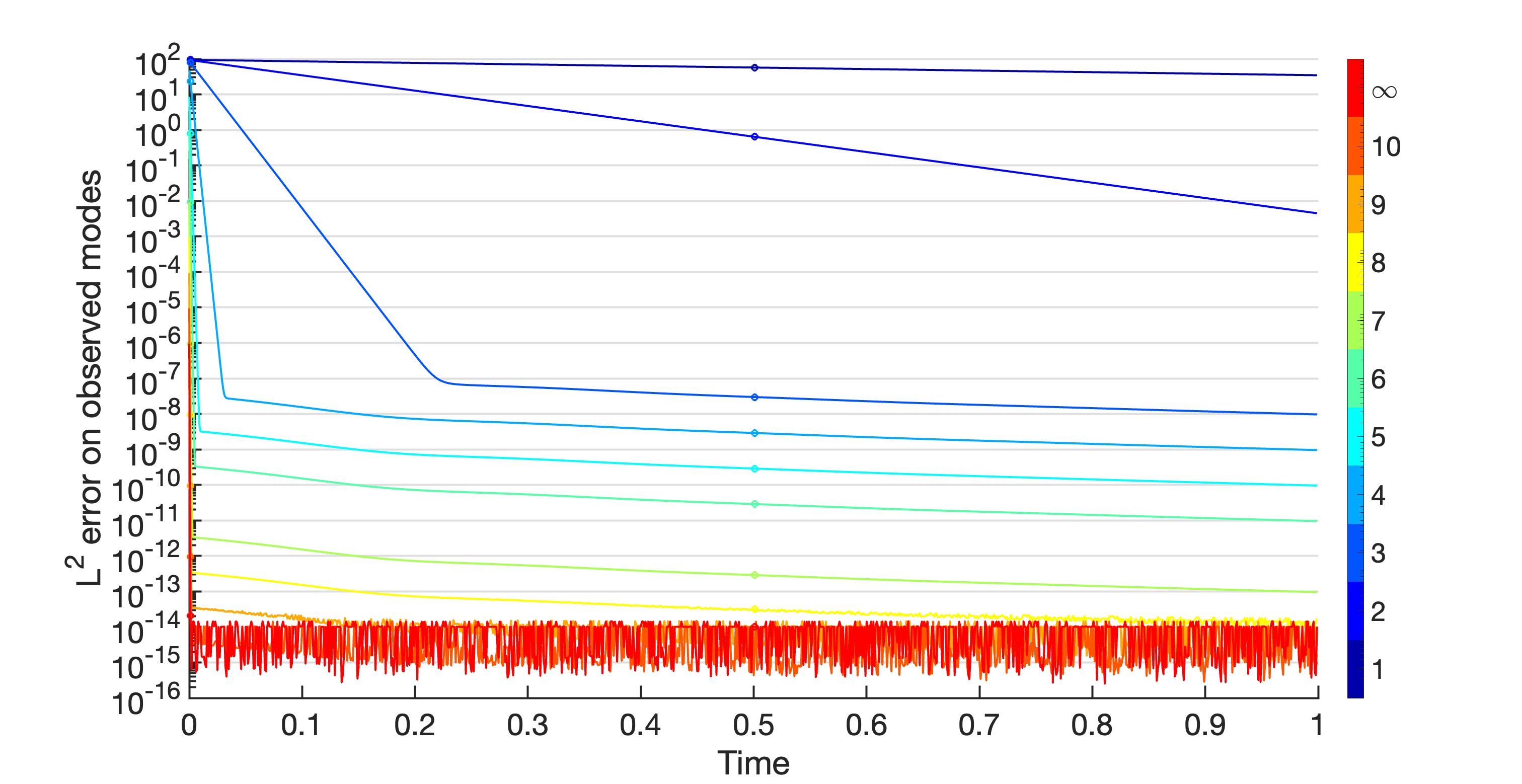}

    % \end{minipage}
    % \caption{Caption}
    \end{subfigure}
    \begin{subfigure}[b]{.45\textwidth}
\centering

        % \begin{minipage}
                % \missingfigure[\width=\textwidth]
            \includegraphics[width=\textwidth,height=5cm]{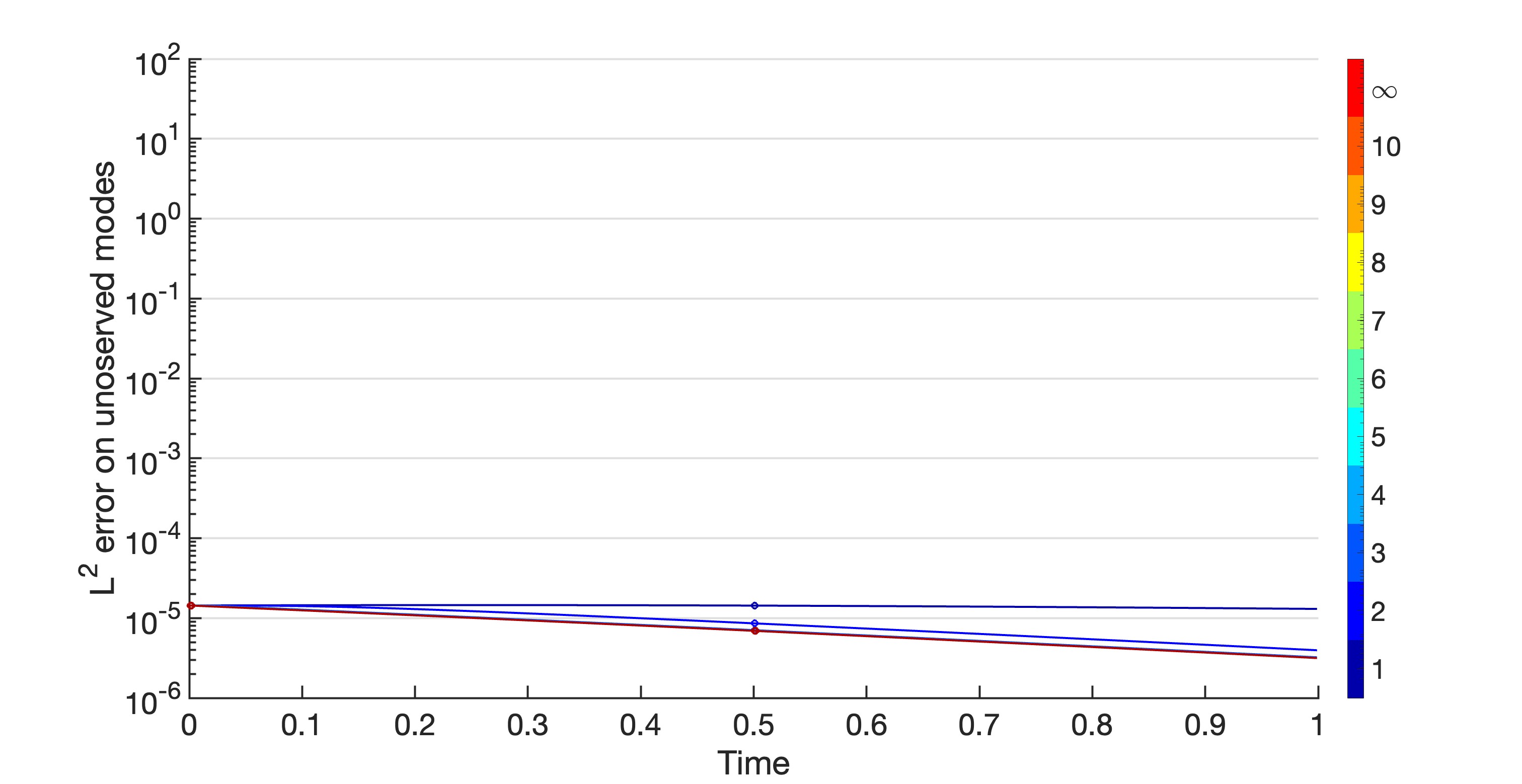}

    % \end{minipage}
    % \caption{Caption}
    \end{subfigure}
%     \begin{subfigure}[b]{.32\textwidth}
% \centering

%         % \begin{minipage}
%                 % \missingfigure[\width=\textwidth]
%             \includegraphics[width=\textwidth,height=5cm]{error_initialized_total.jpg}

%     % \end{minipage}
%     % \caption{Caption}
%     \end{subfigure}

    \caption{\small Error over time between reference and nudging solutions for different $\mu$ values assuming lowest 100 modes observed. Nudging algorithm was initialized with zero initial data. Plotted errors probe infinite-$\mu$ limit and display low mode error (left) and high mode error (right) in $L^2$ norm.  Coloring corresponds to values $\mu = 10^k$, with $k$ indicated by the color bar; $k= \infty$ and $k= -\infty$ correspond to the direct-replacement and zero-nudging regime, respectively.}
    \label{fig:initialized zero}
\end{figure}

\begin{figure}
    \begin{subfigure}[b]{.45\textwidth}
\centering

        % \begin{minipage}
        % \missingfigure[\width=\textwidth]
            \includegraphics[width=\textwidth,height=5cm]{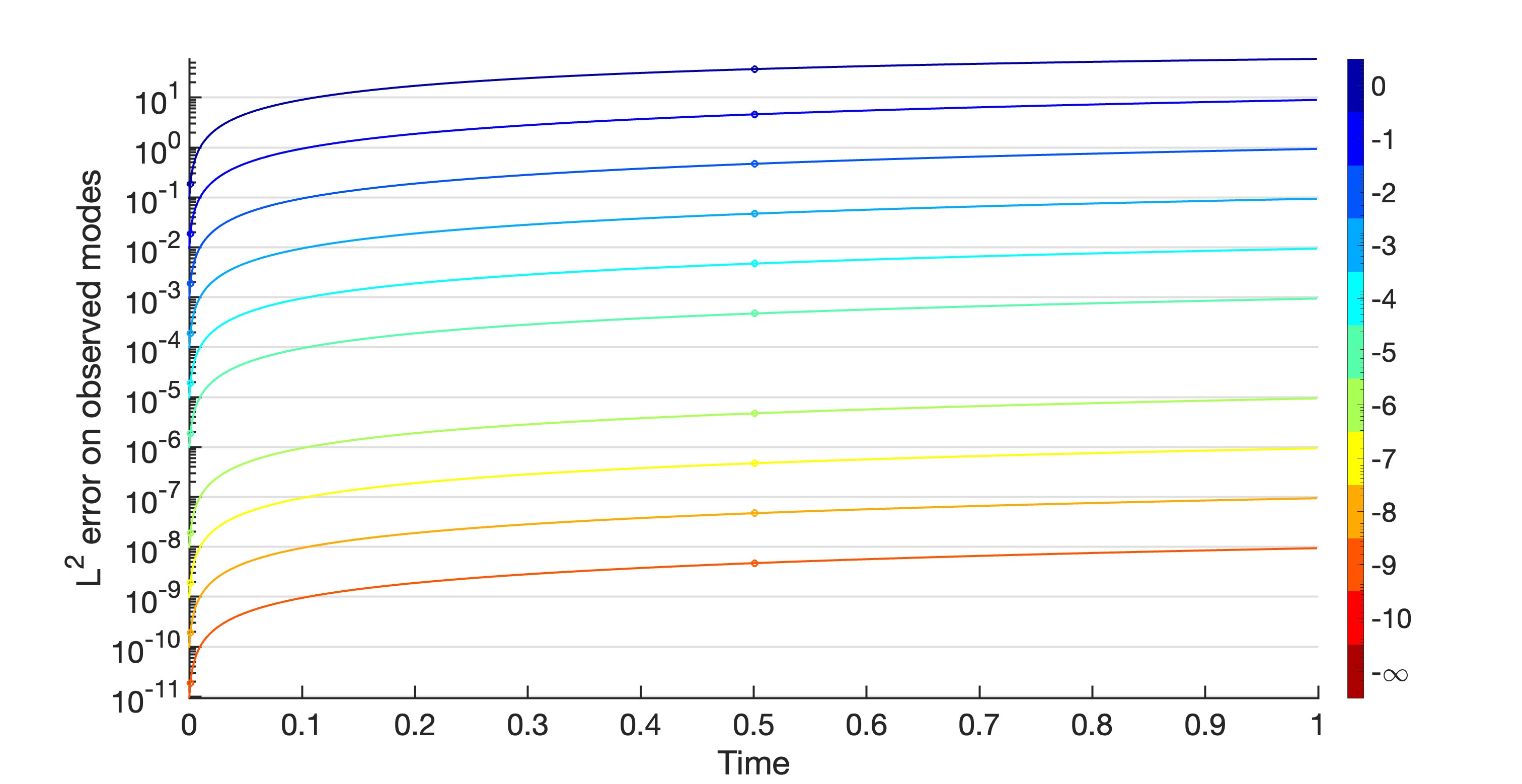}

    % \end{minipage}
    % \caption{Caption}
    \end{subfigure}
    \begin{subfigure}[b]{.45\textwidth}
\centering

        % \begin{minipage}
                % \missingfigure[\width=\textwidth]
            \includegraphics[width=\textwidth,height=5cm]{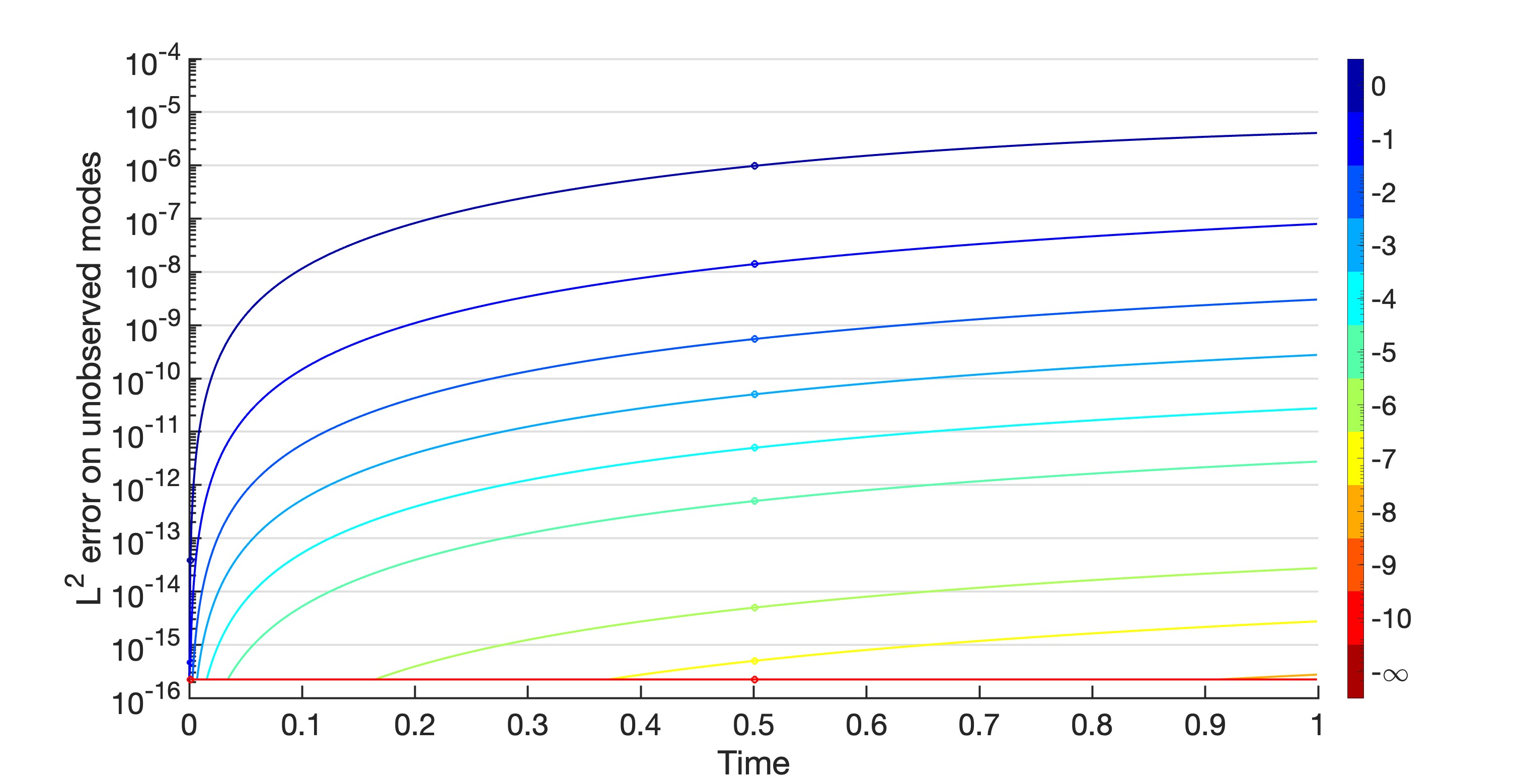}

    % \end{minipage}
    % \caption{Caption}
    \end{subfigure}
%     \begin{subfigure}[b]{.45\textwidth}
% \centering

%         % \begin{minipage}
%                 % \missingfigure[\width=\textwidth]
%             \includegraphics[width=\textwidth,height=5cm]{error_secondary_zero_total.jpg}

%     % \end{minipage}
%     % \caption{Caption}
%     \end{subfigure}

    \caption{\small Error over time for different $\mu$ values assuming lowest 100 modes observed. Nudging algorithm was initialized with zero velocity. Plotted errors probe for zero-nudging limit and represent a splitting of $\norm{\tu - \tv}_{L^2}$, where $\tu$ is the solution of NSE with zero initial data, between low mode errors (left) and high mode errors (right).  Coloring corresponds to values $\mu = 10^k$, with $k$ indicated by the color bar; $k= \infty$ and $k= -\infty$ correspond to the direct-replacement and zero-nudging regime, respectively.}
    \label{fig:zero limit}
\end{figure}

\subsection{Convergence of Nudging to Synchronization - Stochastic Observations}
\label{sect:noisy_obs}

% \subsection{Deterministic Observations}

% \begin{figure}
%     \centering
%     \includegraphics[width=\textwidth,height=5cm]{low_error.jpg}
%     \caption{Error over time for lowest 2 modes observed. }
%     \label{fig:det obs unresolved}
% \end{figure}

% \subsection{Noisy Observations}\label{sect:noisy_obs}

In the comparison of these data assimilation schemes, a vital test is to see how capable they are of handling imperfections in the observational data. To assess this, we utilized observational data polluted with Gaussian white noise.

The modified observational data $u_{obs}$, is formulated as follows:
    \begin{align}\notag
        u_{obs}(t) := P_N u(t) + \eta(t),
    \end{align}
where $\eta$ represents the observational noise injected at each timestep. This noise $\eta$ is constructed as a matrix of random complex coefficients, each corresponding to lower Fourier mode frequencies: 
\begin{align}\label{eq:eta}
    \eta(t) = \sum_{\abs{\vec{k}}=1}^N \hat{\eta}_{\vec{k}}(t) e^{i\pi \vec{k}}.
\end{align}
At each timestep, the coefficients $\hat{\eta}_{\vec{k}}$ are randomly generated as Gaussian white noise variables, each with a standard deviation of $\sigma = 0.1$. It is important to note that $\hat{\eta}_{\vec{k}}$ are complex numbers, having both real and imaginary components generated such that  $\Re\eta, \Im\eta \sim N(0, 0.01)$.

In generating $\eta$ as Gaussian white noise in Fourier space, we took additional steps to ensure the conjugate symmetry of the noise matrix. This is critical for obtaining real coefficients upon employing the inverse Fourier transformation in our numerics. { In particular, we enforce the condition $\hat{\eta}_{\vec{k}} = \overline{\hat{\eta}_{-\vec{k}}}$ for all indices $\vec{k}\in \mathbb{Z}^2$. }
%mirrored version. 
% This was done by mirroring the noise matrix about its center using a combination of rotation and complex conjugation. These operations ensure the Hermitian symmetry of the noise matrix, which ultimately guarantees real values in the spatial domain after applying the inverse Fourier transform.

\begin{algorithm}
\caption{Adaptive $\mu$ Scheme}
\label{alg:Adaptive mu}
\begin{algorithmic}[1] % The number tells where the line numbering should start
\Require $\tilde{u}$, the observations of the reference solution, $u$, polluted by observational noise.
\Require $N$, the number of Fourier modes observed by the interpolant, $P_N$.
\Require $tol$, a tolerance level determining when apply adaptive scheme.
\State $\mu = 1e+5$ \%\% Initialize with large $\mu$
\State $update\_counter = 0$ \%\% Initialize counter to delay slope (exponential decay rate) calculation
\For{$t_i = t_0:\Delta t: T$}
\State $v(t_i) = M(v, P_N(\tilde{u}(t_i), \mu)$ \%\% Evolve DA scheme forward.
\State $E_{obs}(t_i) = \norm{P_N(\tilde{u}(t_i) - v(t_i))}_L^2$
\State $update\_counter = update\_counter+1$
    \If{$update\_counter > 5$} 
        \State $slope = \frac{\log(E_{obs}(t_{i-5})) - \log(E_{obs}(t_i))}{5\Delta t}$
        \If{$slope > tol$}
            \State        $\mu = \mu / 10$
            \State $update\_counter = 0$
        \EndIf
    \EndIf
\EndFor
\end{algorithmic}
\end{algorithm}

Since value of $\mu$ dictates the convergence levels achieved in the {nudging algorithm}, this led us to investigate the effect of using an adaptive value of $\mu$. As one can see in \cref{fig:noisy obs}, the error appears to converge to a static level determined at least in part by $\mu$. However, one can notice that while the resulting convergence for large $\mu$ is overall worse than for small $\mu$ values, the convergence at initial times is noticeably better (see \cref{fig:noisy obs - single}). It appears the value of $\mu$ corresponds to a static error level, yet large values of $\mu$ still correspond to faster convergence to the fixed error level.  Thus, we utilized an adaptive $\mu$ scheme in order to capitalize on both the fast initial convergence of large $\mu$'s and the better overall convergence obtained for smaller $\mu$ values.

\begin{Rmk}
We point out that although $\mu$ drives synchronization, it simultaneously amplifies observational noise, thus leading to substantial loss in precision when $\mu$ is taken too large. This phenomenon was quantified {for the {nudging algorithm} in the presence of observational noise} in the theoretical work {\cite[Theorem 4.1]{BessaihOlsonTiti2015}, where it was} essentially found that the expected error should grow no more than $\mu|\sigma|^2$, {where $\sigma$ represents the variance of the observational noise.} Within our numerical setup, we see that the constant-$\mu$ strategy saturates the analytical error bounds established in {\cite[Theorem 4.1]{BessaihOlsonTiti2015}}. However, the adaptive-$\mu$ strategy proposed here appears to beat these bounds by {a full order of magnitude, which suggests that this gain may be closer in comparison to maximal the error bounds obtained in \cite[Theorem 4.3]{BlomkerLawStuartZygalakis2013} and 
 \cite[Theorem 5.3]{Biswas_Branicki_2024} for the 3DVAR algorithm some choice of covariance operator. See \cref{fig:noisy obs} and \cref{fig:noisy obs - single}. We emphasize that the adaptive-$\mu$ strategy introduced here, is \textit{agnostic} to the covariance structure of the noise, and determined entirely by the errors that are directly observable.} %Since the adaptive-$\mu$ strategy depends only on observable quantities, we view this result as being optimal.
\end{Rmk}

To dynamically adjust the parameter $\mu$ based on the evolution of errors, we opted for a relatively simple approach given in \cref{alg:Adaptive mu}. The main idea is to approximate the { exponential decay rate }of the error on the low modes using the error on the low modes from $5$ {preceding timesteps. We note that the choice of using the last $5$ timesteps is somewhat arbitrary, but {{}we} found it is useful enough for calculating the exponential decay rate} of the observed error. We see that the low-mode error tends to behave as follows:
    \begin{align}\notag
        \norm{P_Nu(t) - P_Nv(t)}_{L^2}^2 \sim \max\left\{ e^{-\kappa t}, C(\mu, \sigma)\right\}, 
    \end{align} 
where $\kappa>0$ is some decay constant, {$\sigma$ represents the variance of the observational noise, and} $C(\mu, \sigma)$ is a constant depending on $\mu$ and $\sigma$. This error tends to decay exponentially until it reaches a level of precision determined by $\mu$ and $\sigma$, after which {it fluctuates, but remains roughly constant}. In the adaptive algorithm, we therefore allow $\mu$ to be sensitive to the observed error and check whether it decays exponentially or remains roughly constant. If the error is roughly constant, then we decrease the value of $\mu$ in order to effect a decrease in the value of $C(\mu, \sigma_O)$. Ultimately, we found that the proposed adaptive $\mu$ scheme increases the overall precision in the long term while maintaining fast initial convergence levels seen with larger $\mu$ values.
We note that we could adjust this algorithm to allow for $\mu$ to increase, however we found this problematic as the value of $\mu$ inflates the observational error, leading to loss of precision if $\mu$ is ever increased in value.
{It is worth noting that this scheme can be readily adapted for a mode-dependent $\mu$ simply by having $\mu$ depend on each individual wave-number and calculating the observed error in \cref{alg:Adaptive mu} on each Fourier mode.}

\begin{Rmk}
In a recent work \cite{CibikFanLaytonSiddiqua2024}, adaptive $\mu$ schemes were also studied in the context of the 2D and 3D NSE. Two such schemes were proposed, the first of which (also called \textit{Algorithm 1}) is similar to the one considered in the present article, but with one notable difference: the scheme considered there allows for the value of $\mu$ to increase when errors have inflated in the next time step, whereas the scheme considered here does not. 

It is important to point out, however, that their tests are carried out in a regime where the {nudging algorithm} is not expected to synchronize with the reference solution. They observe that their adaptive scheme tends to increase $\mu$ in time. Since their observations are perfect, i.e., noise-free, the behavior of their adaptive scheme is consistent with the fact that the {direct-replacement algorithm} should perform the best since the dynamics on the low-modes are exact. This is verified by the results presented in \cref{fig:det:obs:unresolved}.
\end{Rmk}

% \begin{figure}
% \begin{algorithm}[H]
% \caption{Algorithm for Determining $\mu$}
% \KwData{Threshold value, initial value of $\mu$}
% \KwResult{Updated value of $\mu$, recorded switching times}
% \BlankLine
% $\mu = 10000$\;
% $\text{tol} = 0.1$\;
% \While{simulation is running}{
%     Prev_error = Error from 5 timesteps in the past\;
%     Current_error = Error from current timestep\;
%     Diff = log(current_error) - log(prev_error)\;

%     \If{Diff > tol}{
%     $\mu = \mu /10$\;
%     }
% }
% \end{algorithm}
%     % \centering
%     % \includegraphics{}
%     \caption{Caption}
%     \label{fig:enter-label}
% \end{figure}

\begin{figure}
    \begin{subfigure}[b]{.49\textwidth}
\centering

        % \begin{minipage}
            \includegraphics[width=\textwidth,height=5cm]{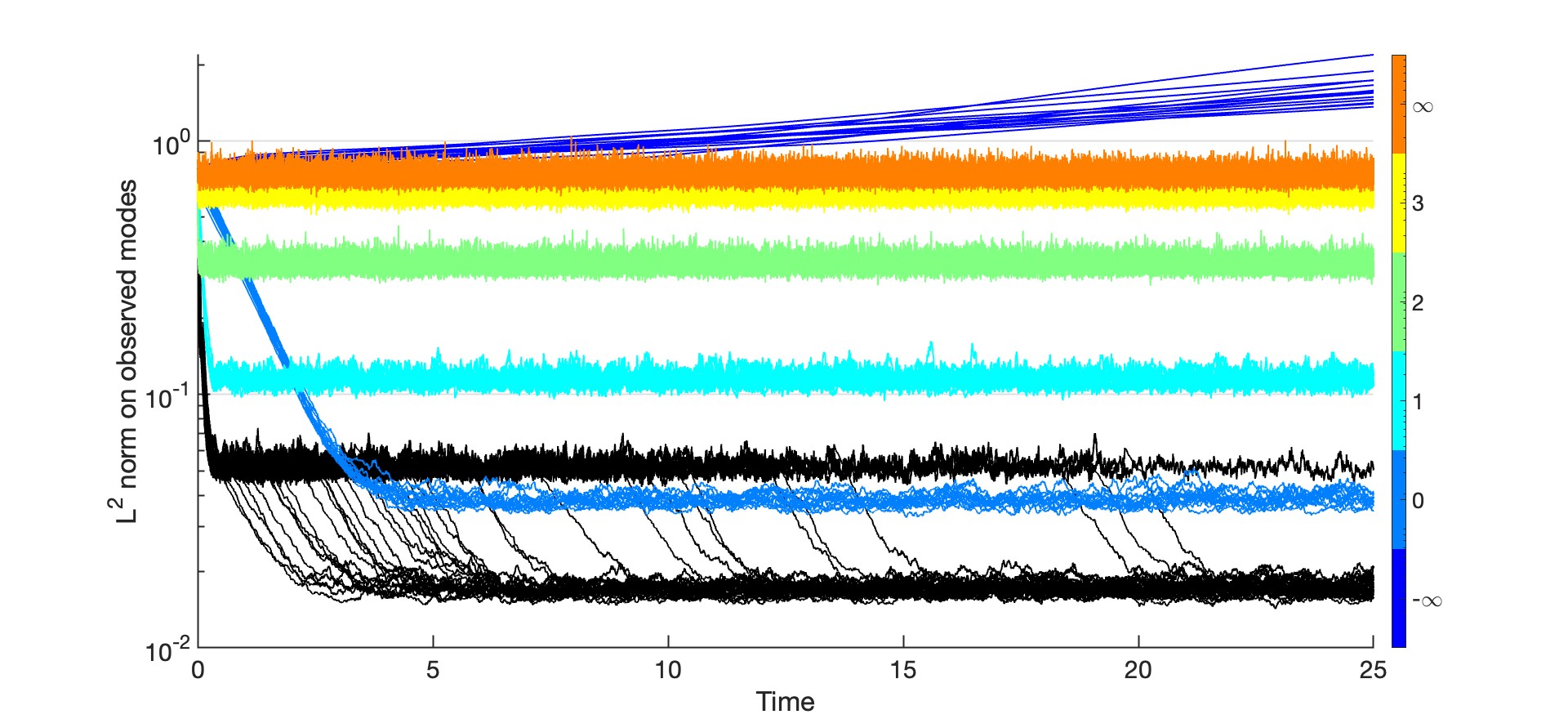}

    % \end{minipage}
    % \caption{Caption}
    \end{subfigure}
    \begin{subfigure}[b]{.49\textwidth}
\centering

        % \begin{minipage}
            \includegraphics[width=\textwidth,height=5cm]{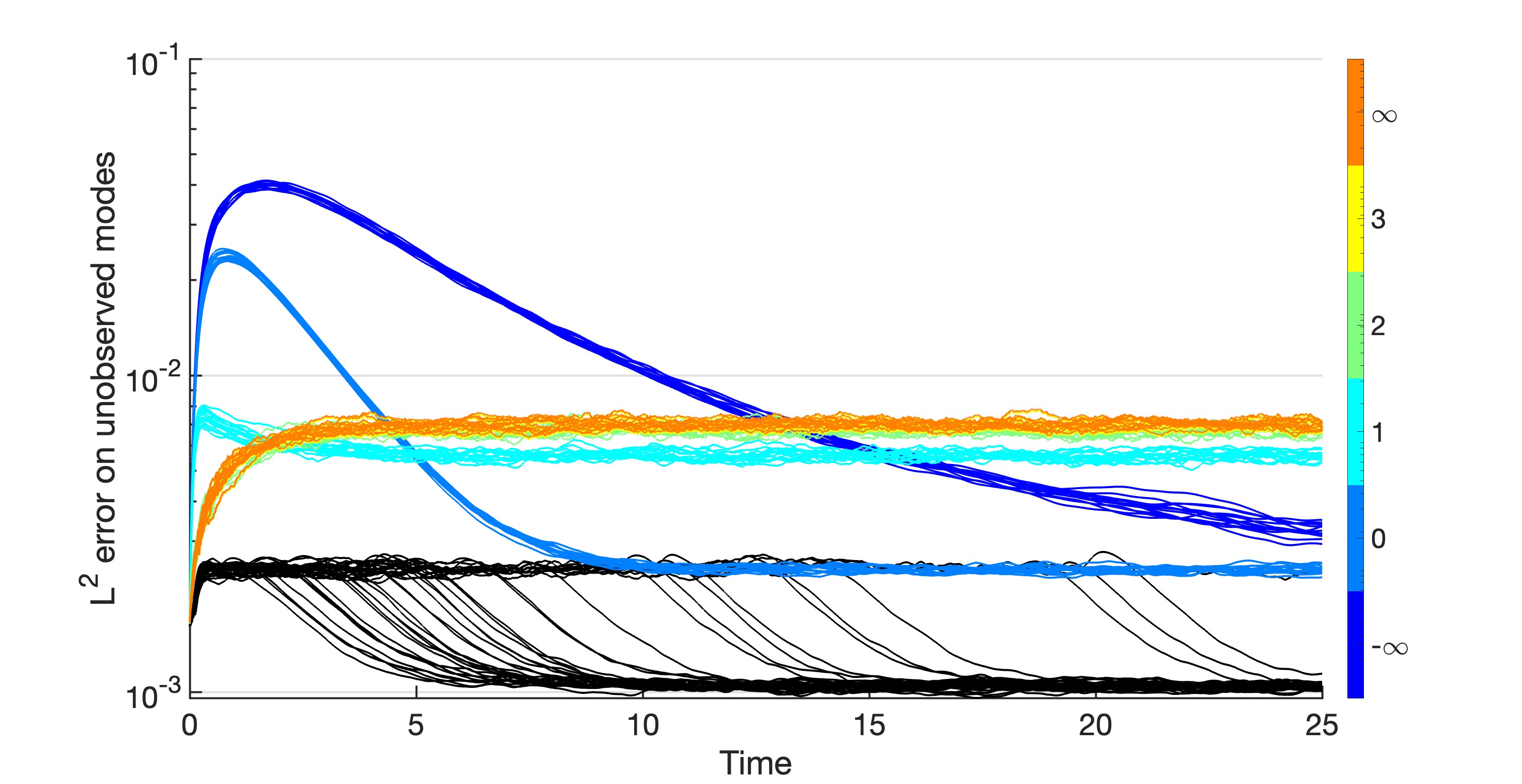}

    % \end{minipage}
    % \caption{Caption}
    \end{subfigure}
    
    \caption{\small Error over time with 30 realizations of noisy observations between reference and nudging solutions for different $\mu$ values. Plotted errors display low mode error (left) and high mode error (right) in $L^2$ norm.  Coloring corresponds to values $\mu = 10^k$, with $k$ indicated by the color bar; $k= \infty$ and $k= -\infty$ correspond to the direct-replacement and zero-nudging regime, respectively. Black lines correspond to the adaptive $\mu$ scheme defined by \cref{alg:Adaptive mu}.}
    \label{fig:noisy obs}
\end{figure}

\begin{figure}
\centering
    \begin{subfigure}[b]{.45\textwidth}
        % \begin{minipage}
            \includegraphics[width=\textwidth,height=5cm]{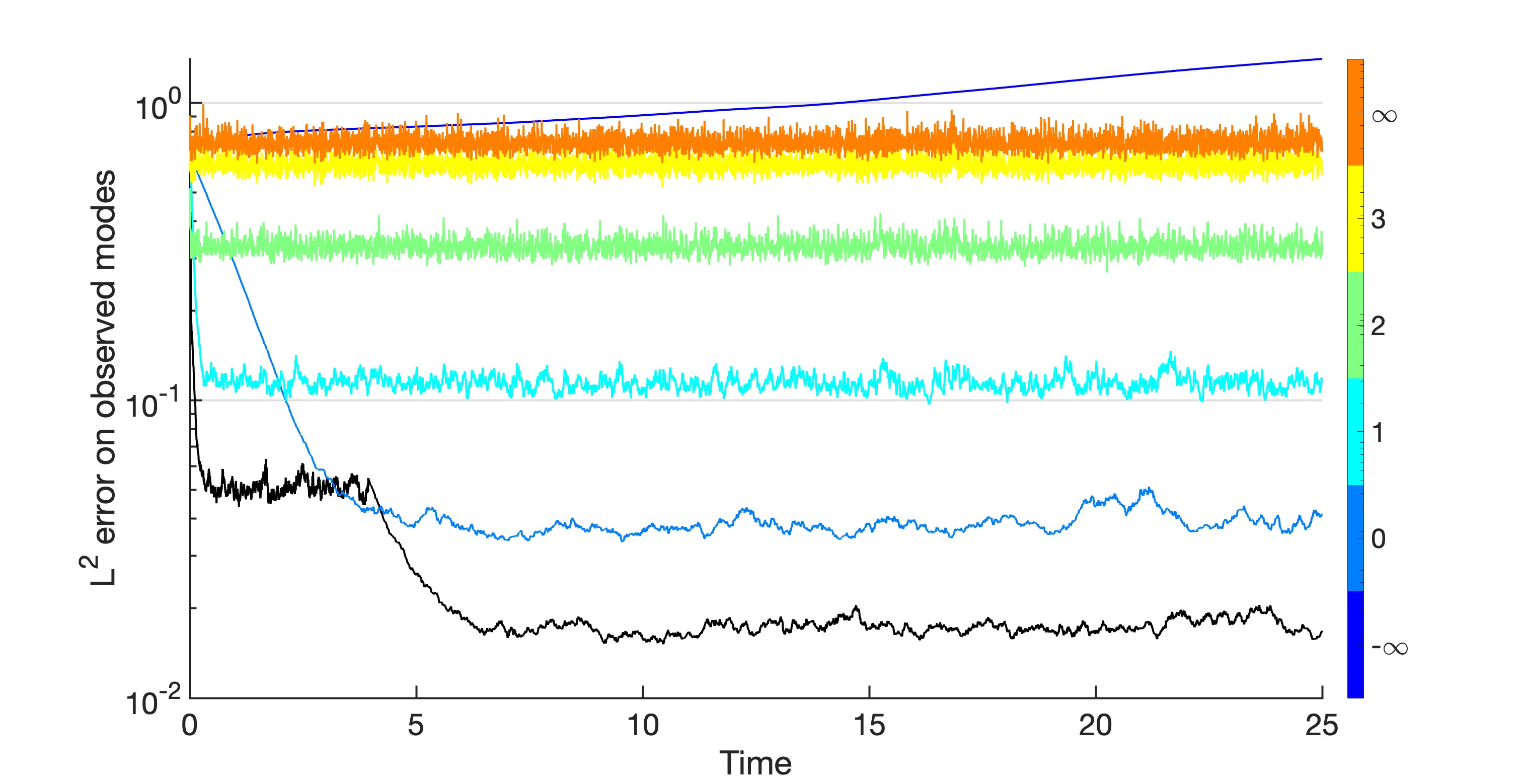}

    % \end{minipage}
    \end{subfigure}
    \begin{subfigure}[b]{.45\textwidth}
\centering

        % \begin{minipage}
            \includegraphics[width=\textwidth,height=5cm]{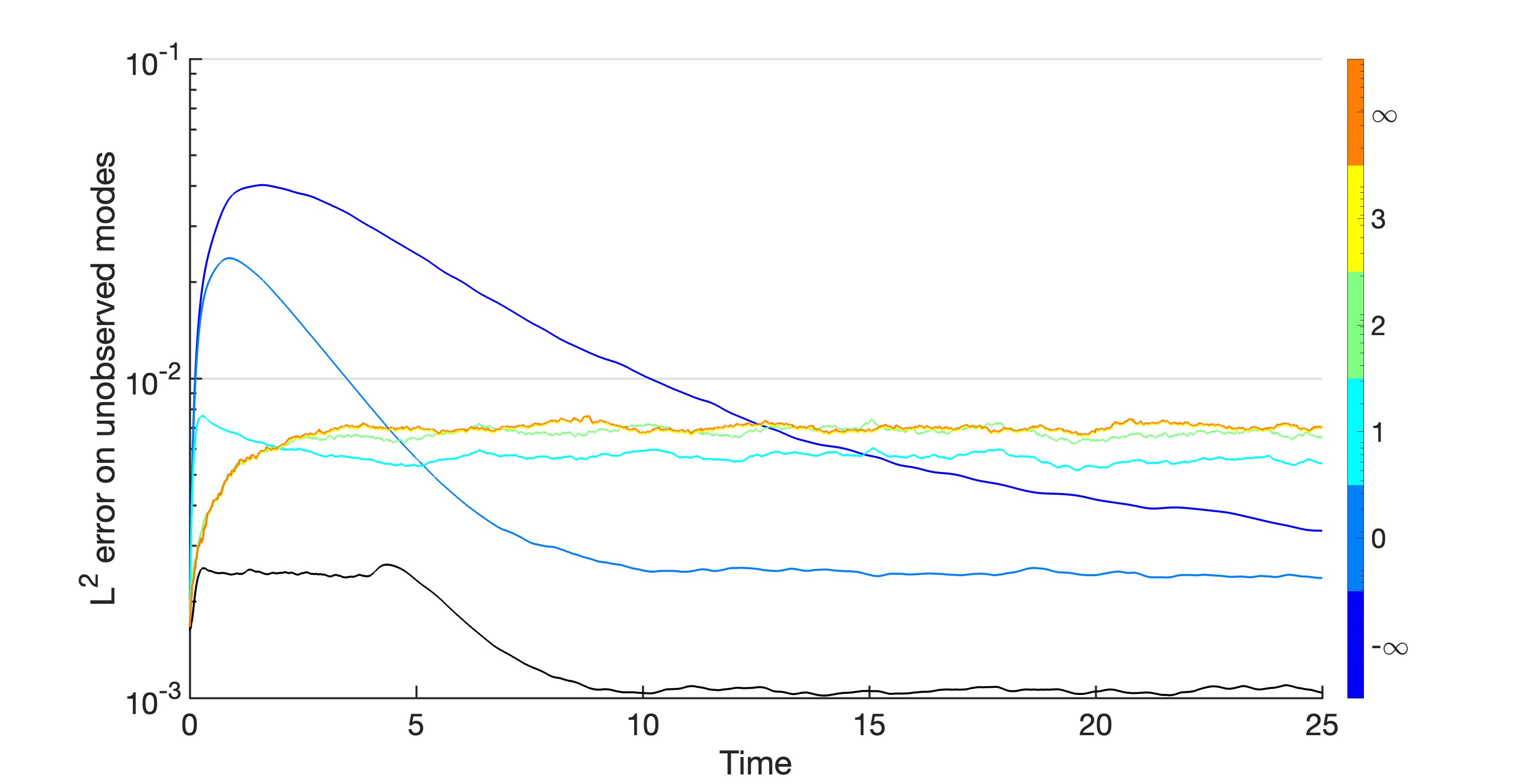}

    % \end{minipage}
    % \caption{}
    \end{subfigure}

        \begin{subfigure}[b]{.45\textwidth}
\centering

        % \begin{minipage}
            \includegraphics[width=\textwidth,height=5cm]{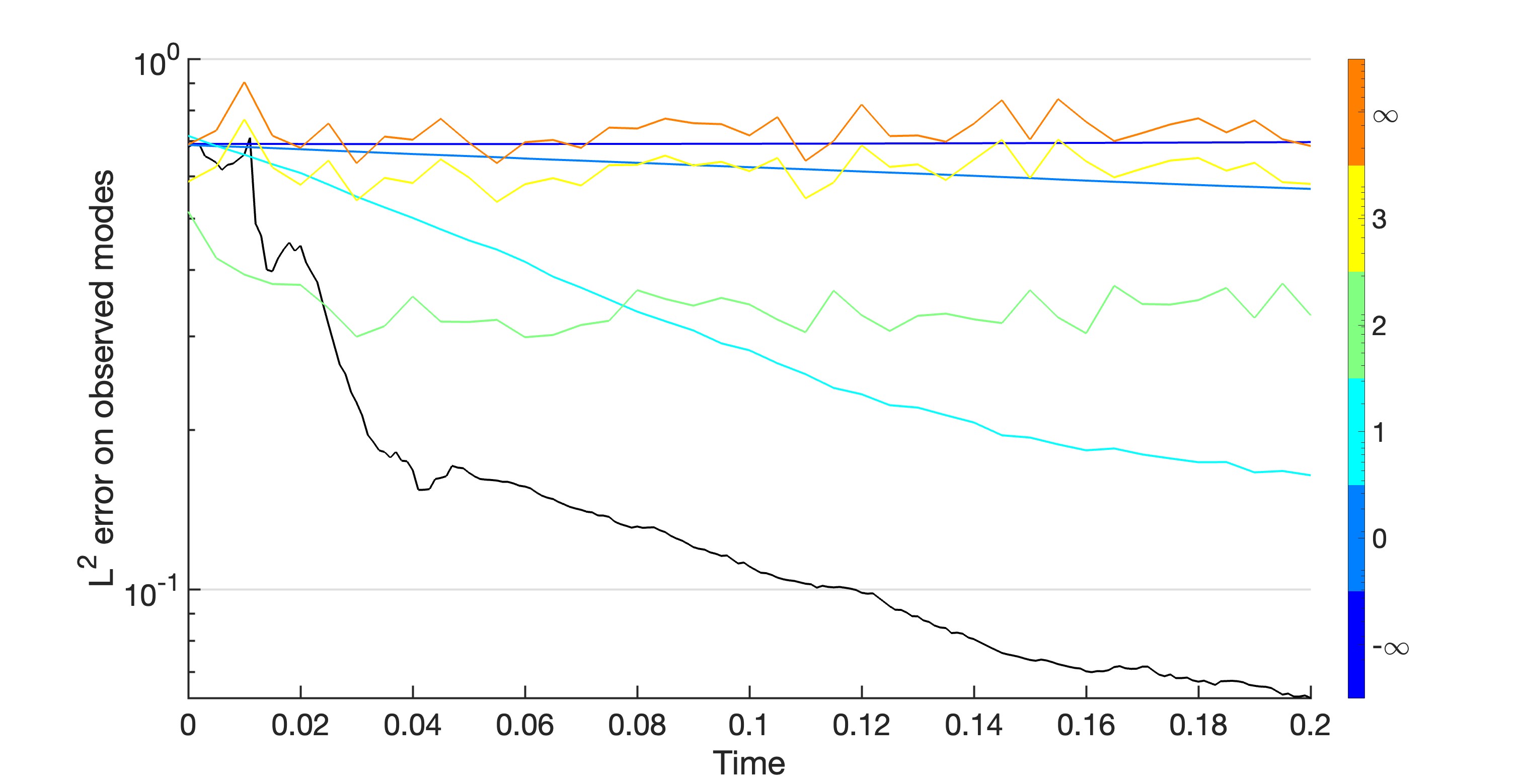}

    % \end{minipage}
    % \caption{}
    \end{subfigure}
    \begin{subfigure}[b]{.45\textwidth}
\centering

        % \begin{minipage}
            \includegraphics[width=\textwidth,height=5cm]{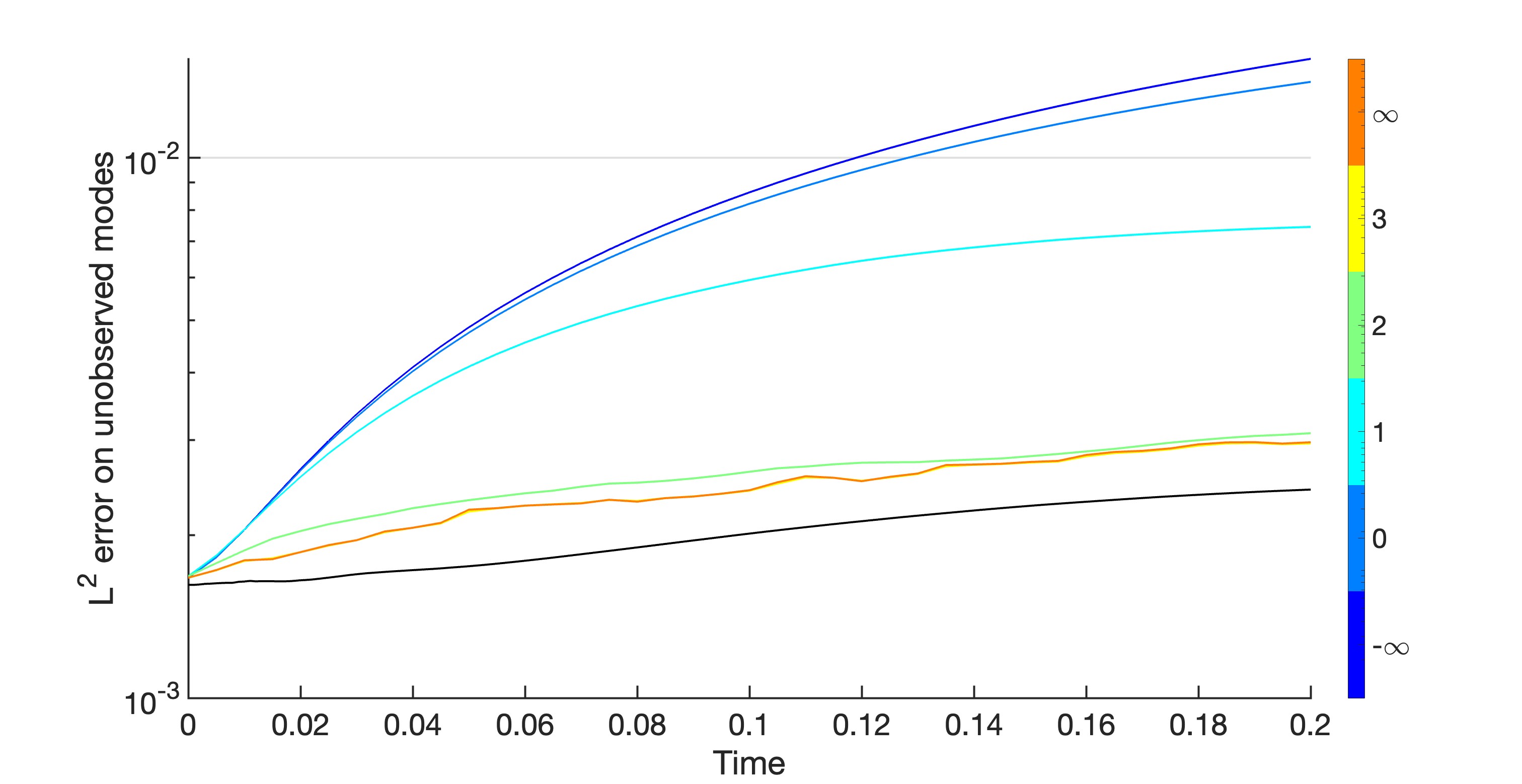}
        % \caption{}

    \end{subfigure}
    \caption{\small 
    Error plots of a single realization from trials shown in \cref{fig:noisy obs}.
    Plotted errors display low mode error (first column) and high mode error (second column) in $L^2$ norm; top row shows global-in-time error; bottom row shows initial transient of error. Coloring corresponds to values $\mu = 10^k$, with $k$ indicated by the color bar; $k= \infty$ and $k= -\infty$ correspond to the direct-replacement and zero-nudging regime, respectively. The black line corresponds to the adaptive $\mu$ scheme defined by \cref{alg:Adaptive mu}.
    }
    \label{fig:noisy obs - single}
    \end{figure}

% \begin{figure}

%     \caption{\small Error between reference and nudging solutions over time for different $\mu$ values for a single trial with observations polluted with measurement error. Errors plotted here are low mode error (left) and high mode error (right), as measured by the $L^2$ norm. Plots are zoomed in to feature early time development of the error.  Coloring corresponds to values $\mu = 10^k$, with $k$ indicated by the color bar; $k= \infty$ and $k= -\infty$ correspond to the direct-replacement and zero-nudging regime, respectively. The black line correspond to an adpative $\mu$ scheme given by \cref{alg:Adaptive mu}.}
%     \label{fig:noisy obs - zoomed}
%     \end{figure}

\FloatBarrier

\subsection*{Acknowledgments} {{}First and foremost, the authors would like to thank the anonymous referees for their generous input and careful reading of the manuscript. Their comments and suggestions gave way to considerable improvements in the article and ultimately to its present state. We are indebted to their time and efforts.} The authors would also like to thank Andrew Stuart and Edriss Titi for their encouragement and for stimulating discussions related to this work.  E.C. was supported in part by the Department of Defense Vannevar Bush Faculty Fellowship, under
ONR award N00014-22-1-2790.  A.F. was supported in part by the National Science Foundation through DMS 2206493. V.R.M. was in part supported by the National Science Foundation through DMS 2213363, DMS 2206491, and DMS 2511403 as well as the Dolciani Halloran Foundation.

\appendix
\section{{The Zero-Nudging Limit: Convergence to the Bjerknes Filter}}\label{sect:limit:zero}
Let $f\in L^\infty(0,\infty;H)$. For $u_0,\tv_0\in V$, and let $u, \tu$ denote the unique global-in-time strong solutions of \eqref{eq:nse:ff} corresponding to $u_0, \tv_0$, respectively, guaranteed by \cref{prop:nse:ball}. {We recall the convention established in the discussion immediately following \cref{thm:nudge:limit} and thus suppose that $u_0\in B_H(\rho_0)\cap B_V(\rho_1)$.} %As in \cref{sect:limit}, we will assume that the reference solution has evolved sufficiently far in time to satisfy the estimates \eqref{est:absorb:L2} at $t=0$, so that without loss of generality we may suppose $t_0=0$ in \eqref{est:absorb:L2} of \cref{prop:nse:ball}. 
Now, given $N>0$, we consider the following set-up:
    \begin{align}
        \frac{d\tv}{dt}+\nu A\tv+B(\tv,\tv)&=f-\mu P_N\tv+\mu P_Nu,\quad \tv(0)=\tv_0, \label{eq:nudge:zero:system:tv}
        \\
        \frac{d\tu}{dt}+\nu A\tu +B(\tu,\tu)&=f,\quad \tu(0)=\tv_0. \label{eq:nudge:zero:system:v}
    \end{align}

\begin{Thm}{\label{thm:nudge:zero:limit}}
Given any $T>0$, one has
    \begin{align}\notag
        \lim\limits_{\mu\goesto {0^+}}\sup\limits_{t\in[0,T]}|\tv(t;\tv_0)-\tu(t;\tv_0)|=0.
    \end{align} 
\end{Thm}

\begin{proof}
Define $\tw = \tv-\tu$, and $w = u-\tu$. Then, $\tw$ satisfies the initial value problem 
    \begin{align}\label{eq:nudge:zero:tw}
        & \frac{d\tw}{dt} - \nu A \tw + B(\tw,\tw)+DB(\tu)\tw = - \mu P_N \tw + \mu P_N w, \qquad \tw(0)=0.
    \end{align} 
Upon taking the $H$--inner product of \eqref{eq:nudge:zero:tw} with $\tw$, we obtain
    \begin{align}\notag
        \frac{1}2\frac{d}{dt}|\tw|^2+\nu\|\tw\|^2+\mu|P_N\tw|^2=-\lp B(\tw,\tu),\tw\rp + \mu\lp P_N w, \tw\rp .
    \end{align}
    By \eqref{est:B:ext} and Young's inequality, we have
    \begin{align*}
        |\lp B(\tw,\tu),\tw\rp|&\leq C_L\|\tw\|\|\tu\||\tw|\leq \nu\|\tw\|^2+\frac{C_L^2}{4\nu}\|\tu\|^2|\tw|^2,
        \\
        \mu|\lp P_N w, \tw\rp|  &\leq \frac{\mu}{2}|P_N w|^2 + \frac{\mu}{2}|\tw|^2.
    \end{align*}
Thus, by \cref{prop:nse:ball} and \eqref{est:abs:ball}, and \eqref{est:Poincare}, we have
    %\begin{align*}\notag
     %   \frac{d}{dt}|\tw|^2+\nu|\tw|^2 + \mu|P_N\tw|^2\leq \nu\frac{C_L^2}{2}\left(\frac{\rho_1}{\nu}\right)^2|\tw|^2 + \mu |P_N w|^2, 
    %\end{align*}
%which can be reduced to 
        \begin{align}\notag
        \frac{d}{dt}|\tw|^2+ \nu\lp1 - \frac{C_L^2}{2}\left(\frac{\rho_1}{\nu}\right)^2 \rp |\tw|^2 \leq \mu |P_N w|^2. 
    \end{align}
By Gr\"onwall's inequality, and since $\tw(0) =0$, we therefore have 
    \begin{align}\notag
        \sup\limits_{t\in[0,T]}|\tw(t)|^2\leq \frac{\mu}{\nu}\left(\frac{\max\{1,e^{-\nu{c_0}T}\}}{|{c_0}|}\right) \sup\limits_{t\in[0,T]} |P_N w{(t)}|^2,
    \end{align}
 where ${c_0}= 1 - \frac{C_L^2}{2}\left(\frac{\rho_1}{\nu}\right)^2$. 
 
 Next, recall that $\tu$ is the solution of the 2D NSE \eqref{eq:nudge:zero:system:v}, while $u$ is the solution of the 2D NSE \eqref{eq:nse}. Thus, by the standard stability argument for the 2D NSE (similar to the argument above), we have 
 \begin{align}\notag
 \sup\limits_{t\in [0,T]}|w(t)|^2 \leq |w_0|^2 \max\{1,e^{-\nu{c_0}T}\}.
 \end{align}
 Thus, 
    \begin{align}\notag
        \sup\limits_{t\in[0,T]}|\tw(t)|^2\leq \frac{\mu}{\nu}\frac{\lp\max\{1,e^{-\nu{c_0}T}\}\rp^2}{|{c_0}|} |w_0|^2,
    \end{align}
which yields the desired conclusion. 
\end{proof}

\newcommand{\etalchar}[1]{$^{#1}$}
\providecommand{\bysame}{\leavevmode\hbox to3em{\hrulefill}\thinspace}
\providecommand{\MR}{\relax\ifhmode\unskip\space\fi MR }
\providecommand{\MRhref}[2]{%
  \href{http://www.ams.org/mathscinet-getitem?mr=#1}{#2}
}
\providecommand{\href}[2]{#2}

%\vfill 

%\hfill

%\hfill

\begin{multicols}{2}
\noindent Elizabeth Carlson\\
{\scriptsize
Department of Computing \& Mathematical Sciences\\
California Institute of Technology\\
Department of Mathematics\\
Oregon State University\footnote{Present address}\\
Web:\url{https://sites.google.com/view/elizabethcarlsonmath}\\
Email: \url{carleliz@oregonstate.edu}\\
}

\noindent Aseel Farhat\\
{\scriptsize
Department of Mathematics \\
University of Virginia \\
Email: \url{af7py@virginia.edu}\\
}

\columnbreak

\noindent Vincent R. Martinez\\
{\scriptsize
Department of Mathematics \& Statistics\\
CUNY Hunter College \\
Department of Mathematics \\
CUNY Graduate Center \\
Department of Computing \& Mathematical Sciences\\
California Institute of Technology\footnote{Present address} \\
Web: \url{http://math.hunter.cuny.edu/vmartine/}\\
Email: \url{vrmartinez@hunter.cuny.edu}, \url{vrm@caltech.edu}\\
}

\noindent Collin Victor\\
{\scriptsize
Department of Mathematics \\
Texas A\&M University \\
Web: \url{https://collinvictor.me/}\\
Email: \url{collin.victor@tamu.edu}\\
}

\end{multicols}

\end{document}